\documentclass[11pt]{article}
\usepackage[totalwidth=13.0cm,totalheight=20.0cm]{geometry}
\usepackage{latexsym,amsthm,amsmath,amssymb,url}
\usepackage[ruled, linesnumbered]{algorithm2e}

\usepackage{tikz,authblk}
\usetikzlibrary{decorations.pathreplacing}

\newcommand{\epX}{\varepsilon}

\newcommand{\GG}[1]{{#1}}

\newcommand{\2}{\vspace{0.15cm}}

\newtheorem{theorem}{Theorem}[section]
\newtheorem{lemma}[theorem]{Lemma}
\newtheorem{corollary}[theorem]{Corollary}

\newtheorem{proposition}[theorem]{Proposition}

\newtheorem{conjecture}[theorem]{Conjecture}
\newtheorem{openProblem}[theorem]{Open Problem}

\begin{document}

\title{Lower Bounds for Maximum Weighted Cut\footnote{\GG{This is a post--publication (SIAM J. Discrete Math. 37(2) 2023) version of this paper, which was produced after we learned about some interesting results on Maximum Weighted Cut, which we were not aware of prior to the publication. }}}
\author[1]{Gregory Gutin}
\author[2,3]{Anders Yeo}
\affil[1]{Department of Computer Science, Royal Holloway, University of London, TW20 0EX, Egham, Surrey, UK, g.gutin@rhul.ac.uk}
\affil[2]{Center for Combinatorics and LPMC, Nankai University, Tianjin 300071, P.R. China}
\affil[3]{IMADA, University of Southern Denmark, Campusvej 55, 5230 Odense, Denmark, andersyeo@gmail.com} 
\affil[4]{Department of Mathematics, University of Johannesburg, Auckland Park, 2006 South Africa}

\date{}
\maketitle

\begin{abstract}
\noindent 
While there have been many results on lower bounds for Max Cut in unweighted graphs,  \GG{there are only few results for lower bounds for Max Cut in weighted graphs}. In this paper, 
we launch an extensive study of lower bounds for Max Cut in weighted graphs. We introduce a new approach for obtaining lower bounds for  Weighted Max Cut. Using it, Probabilistic Method, Vizing's chromatic index theorem,  and other tools,
we obtain several lower bounds for arbitrary weighted graphs, weighted graphs of bounded girth and triangle-free weighted graphs. We pose conjectures and open questions. 
\end{abstract}

\pagestyle{plain}

\section{Introduction}\label{sec:intro}

In this paper $G=(V(G),E(G),w)$ will denote a connected weighted graph with weight function $w:\ E(G) \to \mathbb{R}_+,$ where $\mathbb{R}_+$ is the set of non-negative reals. 
Let $A$ and $B$ be a partition of $V(G).$ Then the {\em cut} $(A,B)$ of $G$ is the bipartite subgraph of $G$ induced by the edges between $A$ and $B$. 
The {\sc Maximum Weighted Cut} problem (MWC) is a well-known {\sf NP}-hard optimization problem on graphs \cite{Karp72}, where given a weighted graph $G$, the aim is to find the maximum weight 
of a cut of $G$. This weight will be denoted by ${\rm mac}(G).$ 

Lower bounds for ${\rm mac}(G)$ are of interest e.g. for designing heuristics and branch-and-bound algorithms for computing ${\rm mac}(G).$
There are many publications where lower bounds on ${\rm mac}(G)$ have been studied. 
However, almost all of them are either for the unweighted case i.e. the weight of every edge equals 1 (see e.g. \cite{Alon96,AloKriSud05,BolSco02,CKLMST20,PolTuz86,Shearer92})
(In what follows, the weight of every edge is an unweighted graph will be equal to 1) or for graphs with integral weights (see e.g. \cite{AlonH98,BolSco02}). 

As far as we know, \GG{there are only three non-trivial lower bounds for the general weighted case: (a) Poljak and Turz{\'{\i}}k \cite{PolTuz86} proved that ${\rm mac}(G)\ge w(G)/2+w(T_{\min})/4,$ where $T_{\min}$ is a minimum weight spanning tree of $G$, (b) Haglin and Venkatesan \cite{HaglinVenk1991}  showed that ${\rm mac}(G)\ge (\frac{1}{2} + \frac{1}{2n})w(G)/2,$ where $n=|V(G)|$,
(c) He, Zhang and Zhang \cite{HeZZ2010} proved that ${\rm mac}(G)\ge (\frac{1}{2} + \frac{1}{400\sqrt{m}})w(G),$ where $m=|E(G)|$.
}
Note that here and in the rest of the paper, for a subgraph $H$ of $G$, $w(H)=\sum_{e\in E(H)}w(e).$ 

In Section \ref{sec:gb}, we introduce a generic lower bound for  ${\rm mac}(G)$ and show that the Poljak-Turz{\'{\i}}k bound can be easily obtained from the generic bound. We prove that unfortunately the lower bound is {\sf NP}-hard to compute. However, the bound can be used to obtain  other lower bounds which are computable in polynomial time, see  Theorems \ref{dfs},  \ref{girth}, \ref{tfree}, \ref{mainProb}, Lemma \ref{mainMatch}, Proposition \ref{match} \GG{and Corollary \ref{HVineq}}.

In Section \ref{sec:ag}, we prove that the Poljak-Turz{\'{\i}}k bound can be improved by replacing a minimum weight spanning tree by a 
DFS tree (i.e. a tree that can be obtained by using a depth-first search algorithm): 
${\rm mac}(G)\ge w(G)/2+w(D)/4,$ where $D$ is a DFS tree of $G.$
Theorem \ref{notANYtree} shows that we cannot replace $D$ in the new bound (called below the {\em DFS bound}) by an arbitrary spanning tree. 
 The DFS tree bound is stronger then the Poljak-Turz{\'{\i}}k bound because while the Poljak-Turz{\'{\i}}k bound requires the spanning tree to be of minimum weight, for the DFS bound we can use an arbitrary DFS tree. We also prove that unfortunately replacing an arbitrary DFS tree with a DFS tree of maximum weight  would make the bound no longer computable in polynomial time unless ${\sf P}={\sf NP}.$ The last result holds even for triangle-tree tree graphs studied later in the paper. We complete Section \ref{sec:ag} by observing another new bound: ${\rm mac}(G)\ge (w(G)+w(M))/2,$ where $M$ is a maximum weight matching of $G.$ 
 
In Section \ref{sec:boundedgirth}, we study graphs of bounded girth. The girth of {a} graph is the length of its shortest cycle. We show that if the girth $g$ of $G$ is even then {${\rm mac}(G)\ge \frac{w(G)}{2} + \frac{g-1}{2g} w(D_{\max})$, where $D_{\max}$ is a maximum weight DFS tree of $G$.} This bound can be extended to the case where the girth is odd by replacing $g$ with $g-1.$ We also prove that when $G$ is  triangle-free then the Poljak-Turz{\'{\i}}k bound can be improved as follows: ${\rm mac}(G)\ge w(G)/2+w(T_{\max})/4,$ where $T_{\max}$ is a maximum weight spanning tree. This is in sharp contrast with Theorem \ref{notANYtree}, which shows that the last bound does not hold  for arbitrary graphs. 
{Note that $w(T_{\max})$ can be computed in polynomial time while it is {\sf NP}-hard to compute $w(D_{\max}).$
We complete the section by a conjecture that ${\rm mac}(G)\ge \frac{w(G)}{2} +\frac{3}{8} \cdot w(T)$ for a triangle-free graph $G$ and  a spanning tree $T$ of $G.$ 

Section \ref{sec:conj} is devoted to triangle-free graphs $G$ with bounded maximum degree $\Delta(G).$ In Subsection \ref{sec:3} we study triangle-free graphs $G$ with $\Delta(G)\le 3.$ Inspired by the result of Bondy and Locke \cite{BoLO86} that a triangle-free graph $G$ with $\Delta(G) \leq 3$ has a  bipartite subgraph with at least $\frac{4}{5}|E(G)|$ edges, we conjecture that ${\rm mac}(G)\ge 4w(G)/5$ for a weighted triangle-free graph $G$ with $\Delta(G) \leq 3$ (see Conjecture~\ref{MAINconjII}). 
Theorem \ref{mainProb} proved in Section \ref{sec:proofs} shows that ${\rm mac}(G) \geq \frac{8}{11} \cdot w(G)$ for a weighted triangle-free graph $G$ with $\Delta(G) \leq 3.$  Theorem \ref{mainProb} allows us to prove Theorem \ref{mainProbTree} (also proved in Section \ref{sec:proofs}) which states that ${\rm mac}(G) \geq \frac{w(G)}{2} + 0.3193 \cdot w(T)$ for a triangle-free graph $G$ with $\Delta(G) \leq 3$ and a spanning tree $T$ of $G.$ We show that Conjecture~\ref{MAINconjII} implies the conjecture of Section \ref{sec:boundedgirth} for triangle-free graphs of maximum degree at most 3 as well as Conjecture~\ref{5cycles}, which states that every triangle-free graph $G$ with $\Delta(G)\le 3$ has an edge set $E'$ such that every 5-cycle of $G$ contains exactly one edge from $E'.$ Thus, if Conjecture~\ref{MAINconjII}  holds, it implies a somewhat unexpected structural result for unweighted graphs.

Subsection \ref{sec:Delta} is devoted to triangle-free graphs $G$ with maximum degree bounded by arbitrary $\Delta.$ The main results of this section are Theorems \ref{Shearer} and \ref{th:Y}, which give different bounds of the type ${\rm mac}(G)\ge a_{\Delta}\cdot w(G),$ where $a_{\Delta}$ depends only on $\Delta.$ The proof of Theorem \ref{Shearer} easily follows from results of Shearer \cite{Shearer92}.
The bound of Theorem \ref{th:Y} is stronger than that of Theorem \ref{Shearer} if and only if $\Delta\le 16.$

We conclude the paper in Section \ref{sec:c}.

Our proofs rely in particular on the Probabilistic Method and Vizing's chromatic index theorem.

\section{Generic Bound}\label{sec:gb}

The following theorem is a generic bound, which is used in the next section to obtain new lower bounds for ${\rm mac}(G)$. 
These new lower bounds immediately imply a well-known lower bound of Poljak and Turz{\'{\i}}k \cite{PolTuz86}. 

Let ${\cal B}(G)$ denote the set of bipartite subgraphs $R$ of $G$ such that every connected component of $R$ is an induced subgraph of $G.$ Every graph in ${\cal B}(G)$ is called a $\cal B$-{\em subgraph}  of $G.$

\begin{theorem} \label{gen}
If $R\in {\cal B}(G)$, then ${\rm mac}(G)\ge (w(G)+w(R))/2.$ 
\end{theorem}
\begin{proof}
Let $R_1,R_2,\dots ,R_{\ell}$ be connected components of $R$ and let $X_i,Y_i$ be partite sets of $R_i$, $i\in [\ell].$ For each $i\in [\ell],$ randomly and uniformly assign $X_i$ color 1 or 2 and $Y_i$ the opposite color. 
{Note that this is a proper coloring of $R.$}
Let $A$ be all vertices of color 1 and let $B$ be all vertices of color 2.  Now every edge in $R$ {deterministically} lies in the cut induced by $(A,B)$ and every edge not in $R$ lies in the cut induced by $(A,B)$ with probability 1/2.  Therefore the average weight of the cut $(A,B)$ is  $w(R) + w(E(G)-E(R))/2 = (w(G) +w(R))/2.$ Thus, ${\rm mac}(G)\ge (w(G) +w(R))/2.$ 
\end{proof}
Using the well-known derandomization method of conditional probabilities \cite[Section 15.1]{AloSpe}, given $R\in {\cal B}(G)$, in polynomial time we can find a cut of $G$ of weight at least $(w(G) +w(R))/2.$ Note that, in the definition of ${\cal B}(G)$,  the requirement that every connected component of $R$ is an induced subgraph of $G$, is necessary as otherwise the term $w(E(G)-E(R))/2$ in the 
average weight of the cut $(A,B)$ is incorrect. 

Let $r_{\max}$ be the maximal weight of a $\cal B$-subgraph of $G$. By Theorem \ref{gen}, ${\rm mac}(G)\ge (w(G) + r_{\max})/2.$ Unfortunately,
it is {\sf NP}-hard to compute $r_{\max}$, which follows from the next theorem.

\begin{theorem} \label{NPbs}
Let $H$ be an unweighted graph. It is {\sf NP}-hard to compute the maximum number of edges in a $\cal B$-subgraph of $H.$
\end{theorem}
\begin{proof}
The proof is by reduction from the {\sc Independent Set} problem. In this problem, given a graph $F$ and a natural number $k$,  we are to decide whether $F$ contains an independent set of size at least $k.$ It is well-known that {\sc Independent Set} is {\sf NP}-complete.

{Let $F$ be an instance of the {\sc Independent Set} problem where we want to determine if $F$ has an independent set of size $k$.
Let $n$ be the number of vertices in $F$ and let $Q$ be a set of $n^2$ vertices outside of $F$. 
Construct a new graph $H$ by adding all edges between $F$ and $Q$.
We will show that $H$ has a $\cal B$-subgraph with at least $kn^2$ edges if and only if $F$ contains an independent set of size at least $k$. Indeed, 
if $F$ contains an independent set $I$ of size at least $k$, then $H$ has a $\cal B$-subgraph with at least $kn^2$ edges, as the subgraph induced by $I \cup Q$ is such a subgraph.

Conversely, assume that $H$ has a $\cal B$-subgraph, $B$, with at least $kn^2$ edges. As at most ${n \choose 2} < n^2$ edges in $E(B)$ belong to $E(F)$ 
we must have at least $(k-1)n^2 +1$ edges in $E(B)$ belonging to the cut $(V(F),Q)$.
Now let $C_1,C_2,\ldots,C_l$ be the connected components of $B$ that contain vertices from $Q$ 
and let $X_i = V(F) \cap V(C_i)$ for all $i \in [l]$. Note that for all $i \in [l]$, $X_i$ is an independent set in $F$, as if some edge $uv$ belonged to $F[C_i]$ then $u$, $v$ and any vertex from $Q \cap V(C_i)$ would form a $3$-cycle in $B$. Let $x^{\max} = \max\{|X_i| \; | \; i \in [l] \}$ and note that no vertex in $Q$ is incident with more than $x^{\max}$ edges in $E(B)$. As there are at least $(k-1)n^2 +1$ edges in $E(B)$ belonging to the cut $(V(F),Q)$, this implies that $x^{\max}>k-1$ which implies that $F$ has an independent set
of size at least $k$.
This completes the proof.}
\end{proof}

The following is a simple corollary of Theorem \ref{gen} as a matching in $G$ is a $\cal B$-{\em subgraph}  of $G.$
\begin{proposition}\label{match}
Let $M$ be a maximum weight matching of $G$. Then ${\rm mac}(G)\ge (w(G)+w(M))/2.$ 
\end{proposition}

The bound of Proposition \ref{match} is tight. Indeed, let $n$ be an even positive integer and let $K_n$ be unweighted. Then  clearly ${\rm mac}(K_n)=n^2/4.$ Also, $(w(K_n)+w(M))/2=\frac{1}{2}{n\choose 2}+n/4=n^2/4.$

It is not hard to construct examples of weighted graphs for which the bound of Proposition \ref{match} is larger than the Poljak-Turz{\'{\i}}k bound; e.g. consider a weighted graph with a spanning tree of 
weight zero and at least one edge of positive weight. 

\GG{Proposition \ref{match} allows us to easily prove the weighted version of the lower bound of Haglin and Venkatesan \cite{HaglinVenk1991}.}

\GG{
\begin{corollary}\label{HVineq}
Let $n=|V(G)|$. Then ${\rm mac}(G)\ge (\frac{1}{2}+\frac{1}{2n})w(G).$ 
\end{corollary}
\begin{proof}
Let $G^*$ be a complete graph obtained from $G$ by adding edges of weight zero between non-adjacent vertices in $G$. Then the average weight of an edge in $G^*$ is $2w(G)/(n(n-1))$. Thus, the expectation of the weight of a maximal matching in $G^*$ (with at least $(n-1)/2$ edges) is at least $w(G)/n$. Choose a maximal matching of weight at least $w(G)/n$ and delete from it all zero-weight edges added in the process of constructing $G^*$. We obtain a matching $M$ of $G$ of weight at least $w(G)/n$.
Now the bound of Proposition \ref{match} implies ${\rm mac}(G)\ge (\frac{1}{2}+\frac{1}{2n})w(G).$
\end{proof}
}

\GG{Haglin and Venkatesan \cite{HaglinVenk1991} conjectured that deciding whether for an unweighted $G$, ${\rm mac}(G)\ge (\frac{1}{2}+\frac{1}{n})m$ is  {\sf NP}-hard, where $m=|E(G)|.$}

\section{New Bounds for Arbitrary {Connected} Graphs}\label{sec:ag}

{Throughout this section, $G$ is connected.}
A {\em DFS tree} is a tree constructed by Depth First Search \cite{CLRS09}. 

\begin{theorem}\label{dfs}
If $D$ is a DFS tree of $G$ then ${\rm mac}(G)\ge w(G)/2+w(D)/4.$
\end{theorem}
\begin{proof}
Let $D$ be rooted at vertex $u.$ Let $L_i$ be the set of vertices of $D$ at distance $i$ from $u$ in $D$.
Let $H_i$ be the subgraph of $D$ induced by the set of edges of $D$ 
between vertices $L_i$ and $L_{i+1}.$ {Note that each $L_j$ is an independent set of $G.$
Since $D$ is a DFS tree, $G$ has no cross edges with respect to $D$} i.e. edges $xy$ such that $x$ is not a descendent of $y$ and  $y$ is not a descendent of $x$ {in $D$} \cite{CLRS09}. Hence, each $H_j$ is an induced bipartite subgraph of $G$.
Let $G_1$ be the disjoint union of graphs $H_i$ with odd $i$ and $G_2$ the disjoint union of graphs $H_i$ with even $i.$
{Since $G_j$ ($j\in \{0,1\}$) is a disjoint union of graphs from ${\cal B}(G)$, we have $G_0,G_1\in {\cal B}(G).$}
Hence, by Theorem \ref{gen}, ${\rm mac}(G)\ge (w(G)+w(G_j))/2$ for $j=0,1.$ These bounds and $w(G_1)+w(G_2)=w(D)$ imply ${\rm mac}(G)\ge w(G)/2+w(D)/4.$
\end{proof}

Note that the above bound immediately implies the following corollary, by Poljak and Turz{\'{\i}}k (as $w(D) \geq w(T_{\min})$ when $D$ is a DFS tree and 
$T_{\min}$ is a minimum weight spanning tree of $G$).

\begin{corollary} \label{PT}
Let $T_{\min}$ be a minimum weight spanning tree of $G.$ Then ${\rm mac}(G)\ge w(G)/2+w(T_{\min})/4.$
\end{corollary}

As some DFS trees may have much larger weight than a minimum weight spanning tree of $G$, 
the bound of Theorem \ref{dfs} is, in general, stronger than that of Poljak and Turz{\'{\i}}k. 
The following theorem implies that $D$ cannot be replaced by an arbitrary spanning tree $T$ in Theorem~\ref{dfs}.

\begin{theorem}\label{notANYtree}
Let $\epX{}>0$ be arbitrary.  There exists (infinitely many) edge-weighted graphs $G$ with a spanning tree 
$T$ such that
${\rm mac}(G) < w(G)/2+ \epX{} w(T).$
\end{theorem}
\begin{proof}
Let $W$ be any positive integer strictly larger than $\frac{1}{4\epX{}}.$
Let $l$ be any integer such that:
\[
 l > \frac{W^2}{4W\epX{} -1}.
\]
Let $G = K_{l+1}$ and let $x \in V(G)$ be arbitrary. Let $w:E(G) \rightarrow \mathbb{N}$ be a weight function
such that $w(e)=W$ if $e$ is an edge incident with $x$ and let $w(e)=1$ otherwise. Let $T$ be the spanning star
$K_{1,l}$ with $x$ as the root. Note that $w(T) = Wl$.
Let $\theta = w(G)/2 + \epX{} \; w(T)$ and note that the following holds.
\[
\theta  =  \frac{1}{2} {l \choose 2} + Wl \left( \frac{1}{2} + \epX{} \right)   =  \frac{l(l-1)}{4} + Wl \left( \frac{1}{2} + \epX{} \right)  
\]
We will now bound ${\rm mac}(G)$ from above. Let $(A,B)$ be any partition of $V(G)$ and without loss of generality assume that $x \in A$.
Let $G'$ be the graph obtained from $G$ by replacing $x$ with an independent set, $X$, containing $W$ vertices, such that
$N(v) = V(G) \setminus \{x\}$  for all $v \in X$. Let all edge-weights of $G'$ be one. Note that the weight of the cut
$(X \cup A \setminus \{x\},B)$ in $G'$ is the same as the weight of the cut $(A,B)$ in $G$ and as $(A,B)$ is an arbitrary
cut in $G$ we must have ${\rm mac}(G') \geq {\rm mac}(G)$. As the maximum number of edges in a cut in a graph on $W+l$ vertices is at most
$((W+l)/2)^2$ the following must hold.
\[
{\rm mac}(G) \leq {\rm mac}(G') \leq \left( \frac{W+l}{2} \right)^2 = \frac{(W + l)^2}{4}
\]
Recall that  $l > \frac{W^2}{4W\epX{} -1}$, which implies that $l(4W\epX{} -1) > W^2$.
Adding $2Wl+l^2$ to both sides gives us the following:
\[
l^2 + 4lW \left( \frac{1}{2} + \epX{} \right) - l > 2Wl + W^2 + l^2
\]
Dividing both sides by $4$ and recalling our bounds for $\theta$ and ${\rm mac}(G)$ implies the following:
\[
{\rm mac}(G) \leq \frac{2Wl + W^2 + l^2}{4} <  \frac{l(l-1)}{4} + lW \left( \frac{1}{2} + \epX{} \right) = \theta 
\]
\end{proof}

Unfortunately, it is {\sf NP}-hard to compute the maximum weight of a DFS tree in a weighted graph even in the class of {\em triangle-free} graphs, i.e. graphs that do not contains $K_3$ as a subgraph.

\begin{theorem}\label{NPdfs}
It is {\sf NP}-hard to compute the maximum weight of a DFS tree in a weighted triangle-free graph.
\end{theorem}
\begin{proof}
We will reduce from the Hamilton $(x,y)$-path problem proved to be {\sf NP}-complete in \cite{ItPaSz82} for triangle-free graphs (it is {\sf NP}-complete  already for grid graphs). 
In this problem, given a triangle-free graph $G$ and two vertices $x,y$ of $G$, we are to decide whether $G$ has a Hamilton path with end-vertices $x$ and $y.$ 
Given such a graph $G$ and $x,y \in V(G)$ we create a new weighted graph $G'$ as follows {(see Figure  \ref{Thm34fig})}.

Let $V(G)=\{v_1,v_2,\ldots,v_n\}$ and let $V(G') = V(G) \cup Z \cup \{z\}$, where $Z=\{z_1,z_2,\ldots,z_n\}$.
Let $E(G')=E(G) \cup \{ z_i v_i, z z_i \; | \; i=1,2,\ldots,n \}$. Note that $G'$ is triangle-free.
Assume that $v_n=y$ and let the weight of all edges $z_i v_i$ be 1 for $i=1,2,\ldots,n-1$ and let 
the weight of all other edges in $G'$ be 2. Note that the maximum weight of a spanning tree in $G'$ is at most $2(|V(G')|-1).$

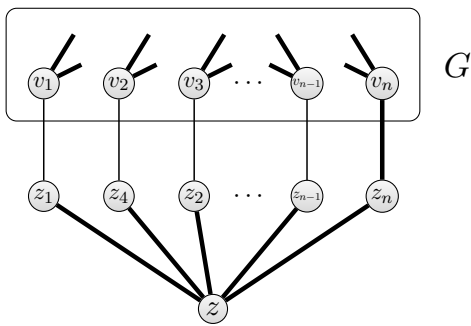
\begin{figure}[hbtp]
\begin{center}
\tikzstyle{vertexDOT}=[scale=0.25,circle,draw,fill]
\tikzstyle{vertexY}=[circle,draw, top color=gray!5, bottom color=gray!30, minimum size=11pt, scale=0.8, inner sep=0.99pt]
\tikzstyle{vertexZ}=[circle,draw, top color=gray!5, bottom color=gray!30, minimum size=11pt, scale=0.5, inner sep=0.49pt]
\tikzstyle{vertexW}=[circle,draw, top color=gray!5, bottom color=gray!30, minimum size=11pt, scale=1.0, inner sep=1.1pt]

\begin{tikzpicture}[scale=0.5]
\node (v1) at (1,7) [vertexY] {$v_1$};
\node (v2) at (3,7) [vertexY] {$v_2$};
\node (v3) at (5,7) [vertexY] {$v_3$};
\node (v4) at (8,7) [vertexZ] {$v_{n-1}$};
\node (v5) at (10,7) [vertexY] {$v_n$};
\draw [rounded corners] (0,6) rectangle (11,9);   

\node (z1) at (1,4) [vertexY] {$z_1$};
\node (z2) at (3,4) [vertexY] {$z_4$};
\node (z3) at (5,4) [vertexY] {$z_2$};
\node (z4) at (8,4) [vertexZ] {$z_{n-1}$};
\node (z5) at (10,4) [vertexY] {$z_n$};

\node (z) at (5.5,1) [vertexW] {$z$};

\draw [line width=0.06cm] (v1) -- (2,7.5);
\draw [line width=0.06cm] (v1) -- (1.8,8.3);
\draw [line width=0.06cm] (v2) -- (4,7.5);
\draw [line width=0.06cm] (v2) -- (3.8,8.3);
\draw [line width=0.06cm] (v3) -- (6,7.5);
\draw [line width=0.06cm] (v3) -- (5.8,8.3);

\draw [line width=0.06cm] (v4) -- (7,7.5);
\draw [line width=0.06cm] (v4) -- (7.2,8.3);
\draw [line width=0.06cm] (v5) -- (9,7.5);
\draw [line width=0.06cm] (v5) -- (9.2,8.3);

\draw [line width=0.02cm] (v1) -- (z1);
\draw [line width=0.02cm] (v2) -- (z2);
\draw [line width=0.02cm] (v3) -- (z3);
\draw [line width=0.02cm] (v4) -- (z4);
\draw [line width=0.06cm] (v5) -- (z5);

\draw [line width=0.06cm] (z) -- (z1);
\draw [line width=0.06cm] (z) -- (z2);
\draw [line width=0.06cm] (z) -- (z3);
\draw [line width=0.06cm] (z) -- (z4);
\draw [line width=0.06cm] (z) -- (z5);
  \node [scale=1.2] at (12,7.5) {$G$};
  \node [scale=0.9] at (6.5,7) {$\cdots$};
  \node [scale=0.9] at (6.5,4) {$\cdots$};
\end{tikzpicture}
\end{center}
\caption{The graph $G'$ in Theorem~\ref{NPdfs}. The thick edges have weight 2 and the thin edges have weight 1.}
\label{Thm34fig}
\end{figure}

We first show that there is a DFS tree rooted at $x$ of weight $2(|V(G')|-1)$ in $G'$ if and only if $G$ contains a Hamilton $(x,y)$-path.
Assume that $G$ contains a Hamilton $(x,y)$-path, $P$. Then $E(P)$ together with the edge $v_n z_n$ and all edges from $z$ to $Z$ form 
a DFS tree rooted at $x$ where all edges have weight 2.

Conversely assume that there is a DFS tree, {$T'$}, in $G'$, rooted in $x$, 
of weight $2(|V(G')|-1)$. This implies that all edges in {$T'$} have weight 2. As no edge of weight 1 is used 
{ in $T'$, the only edge between $T=T'[V(G)]$ and $T'-V(G)$ is $yz_n$. 
Thus, $T$ is a tree. Note that $T$ does not have  a vertex $v_j$ with $j<n$ as a leaf since otherwise {$T'$ would have edge $v_jz_j$, which is impossible}. 
Thus, $T$ is just a Hamilton path of $G$ from $x$ to $y.$}

This shows that there is a DFS tree rooted at $x$ of weight $2(|V(G')|-1)$ if and only if $G$ contains a Hamilton $(x,y)$-path.

Now create the graph $G^*$ by taking two copies, $G_1'$ and $G_2'$, of $G'$ and adding an edge of weight 2 between the copy of vertex $x$ in $G_1'$ and the copy of vertex $x$ in $G_1''.$
Note that $G^*$ is triangle-free.
 If there is a Hamilton $(x,y)$-path in $G$, then as shown above, there is a DFS tree in $G'$ rooted at $x$, 
where all edges have weight 2, which implies that there is DFS tree in $G^*$ where all edges have weight 2 (rooted in one of 
the copies of $x$). 

Conversely assume that there is a DFS tree, $T^*$, in $G^*$ where all edges have weight 2. Without loss 
of generality, assume the root of $T^*$ lies in $G_1'$. Then $T^*[V(G_2')]$ is a DFS tree in $G_2'$ rooted at its copy of $x$, which by the above implies that
there is a Hamilton $(x,y)$-path in $G$. 

Therefore, if we can decide whether there is a DFS-tree in $G^*$ (which is triangle-free) only containing edges of weight 2, then we can decide whether $G$ has a Hamilton $(x,y)$-path. This completes our proof.
\end{proof}



\section{Bounded Girth Families of {Connected} Graphs}\label{sec:boundedgirth}

{Throughout this section, $G$ is connected.}
The {\em girth} of $G$ is the minimum number of edges in a cycle of $G$. The {\em depth} of a rooted tree is the maximum number of edges in a path from the root to a leaf. 

\begin{theorem}\label{girth}
Let $k$ be a positive even integer. If the girth of $G$ is at least $k$, then ${\rm mac}(G)\ge \frac{w(G)}{2} + \frac{k-1}{2k} w(D)$, for every DFS tree $D$ in $G$.
\end{theorem}
\begin{proof}
Let $G$ be any graph with girth at least $k$ and let $D$ be any DFS-tree in $G$. Let $r$ denote the root of $D$ and 
let $L_i$ be the set of vertices of $D$ reached from $r$ by a path with $i$ edges. Note that $L_0 = \{r\}$.

For $j=0,1,\ldots,k-1$, let $G_j$ be the subgraph induced by the edges of $T$ minus those between $L_i$ and $L_{i+1}$ for
all $i = j$ (mod $k$). Also add to $G_j$ the edges of $G-V(D)$ linking vertices in the same connected component of $G_j$ 
{(see Figure~\ref{Thm41fig} for an illustration of the $G_i$'s).}  
Since $D$ is a DFS tree, $D$ has no cross edges i.e. edges $xy$ such that $x$ is not a descendent of $y$ and  $y$ is not a descendent of $x$ \cite{CLRS09}.
As the girth of $G$ is at least $k,$ every connected component of $G_j$ consists of a tree of depth at most $k-1$ plus possibly some 
edges from the leaves to its root (if the leaves are at distance $k-1$ from the root). Since every connected component of $G_j$ is an induced subgraph of $G$ and $k$ is even, we note that  
every connected component of $G_j$ belongs to ${\cal B}(G)$ and and by Theorem~\ref{gen}, ${\rm mac}(G)\ge (w(G)+w(G_j))/2$ for every $j=0,1,\ldots,k-1.$

\begin{figure}[hbtp]
\begin{center}
\begin{tabular}{|c||c|c|c|c|} \hline
\tikzstyle{vertexDOT}=[scale=0.25,circle,draw,fill]
\tikzstyle{vertexY}=[circle,draw, top color=gray!5, bottom color=gray!30, minimum size=11pt, scale=0.8, inner sep=0.99pt]
\tikzstyle{vertexZ}=[circle,draw, top color=gray!5, bottom color=gray!30, minimum size=11pt, scale=0.5, inner sep=0.49pt]
\tikzstyle{vertexW}=[circle,draw, top color=gray!5, bottom color=gray!30, minimum size=11pt, scale=1.0, inner sep=1.1pt]

\begin{tikzpicture}[scale=0.45]
\node (v1) at (2,11) [vertexY] {$v_1$};
\node (v2) at (2,9) [vertexY] {$v_2$};
\node (v3) at (1,7) [vertexY] {$v_3$};
\node (v4) at (1,5) [vertexY] {$v_4$};
\node (v5) at (1,3) [vertexY] {$v_5$};
\node (v6) at (1,1) [vertexY] {$v_6$};
\node (v7) at (2.2,3) [vertexY] {$v_7$};
\node (v8) at (3,7) [vertexY] {$v_8$}; \node (v9) at (3,5) [vertexY] {$v_9$};

\draw [line width=0.06cm] (v1) -- (v2);
\draw [line width=0.06cm] (v2) -- (v3);
\draw [line width=0.06cm] (v3) -- (v4);
\draw [line width=0.06cm] (v4) -- (v5);
\draw [line width=0.06cm] (v5) -- (v6);
\draw [line width=0.06cm] (v4) -- (v7);
\draw [line width=0.06cm] (v2) -- (v8);
\draw [line width=0.06cm] (v8) -- (v9);

\draw [line width=0.01cm] (v2) to [out=210, in=120]  (v6);
\draw [line width=0.01cm] (v2) -- (v7);
\draw [line width=0.01cm] (v1) to [out=300, in=50]  (v9);

\node [scale=1.2] at (2,-0.5) {(a)};
\node [scale=0.9] at (2,11.7) {\mbox{ }};
\end{tikzpicture}
&
\tikzstyle{vertexDOT}=[scale=0.25,circle,draw,fill]
\tikzstyle{vertexY}=[circle,draw, top color=gray!5, bottom color=gray!30, minimum size=11pt, scale=0.8, inner sep=0.99pt]
\tikzstyle{vertexZ}=[circle,draw, top color=gray!5, bottom color=gray!30, minimum size=11pt, scale=0.5, inner sep=0.49pt]
\tikzstyle{vertexW}=[circle,draw, top color=gray!5, bottom color=gray!30, minimum size=11pt, scale=1.0, inner sep=1.1pt]

\begin{tikzpicture}[scale=0.45]
\node (v1) at (2,11) [vertexY] {$v_1$};
\node (v2) at (2,9) [vertexY] {$v_2$};
\node (v3) at (1,7) [vertexY] {$v_3$};
\node (v4) at (1,5) [vertexY] {$v_4$};
\node (v5) at (1,3) [vertexY] {$v_5$};
\node (v6) at (1,1) [vertexY] {$v_6$};
\node (v7) at (2.2,3) [vertexY] {$v_7$};
\node (v8) at (3,7) [vertexY] {$v_8$}; \node (v9) at (3,5) [vertexY] {$v_9$};

\draw [line width=0.06cm] (v2) -- (v3);
\draw [line width=0.06cm] (v3) -- (v4);
\draw [line width=0.06cm] (v4) -- (v5);
\draw [line width=0.06cm] (v4) -- (v7);
\draw [line width=0.06cm] (v2) -- (v8);
\draw [line width=0.06cm] (v8) -- (v9);

\draw [line width=0.01cm] (v2) -- (v7);

\node [scale=1.2] at (2,-0.5) {$G_0$};
\node [scale=0.9] at (2,11.7) {\mbox{ }};
\end{tikzpicture}
&
\tikzstyle{vertexDOT}=[scale=0.25,circle,draw,fill]
\tikzstyle{vertexY}=[circle,draw, top color=gray!5, bottom color=gray!30, minimum size=11pt, scale=0.8, inner sep=0.99pt]
\tikzstyle{vertexZ}=[circle,draw, top color=gray!5, bottom color=gray!30, minimum size=11pt, scale=0.5, inner sep=0.49pt]
\tikzstyle{vertexW}=[circle,draw, top color=gray!5, bottom color=gray!30, minimum size=11pt, scale=1.0, inner sep=1.1pt]

\begin{tikzpicture}[scale=0.45]
\node (v1) at (2,11) [vertexY] {$v_1$};
\node (v2) at (2,9) [vertexY] {$v_2$};
\node (v3) at (1,7) [vertexY] {$v_3$};
\node (v4) at (1,5) [vertexY] {$v_4$};
\node (v5) at (1,3) [vertexY] {$v_5$};
\node (v6) at (1,1) [vertexY] {$v_6$};
\node (v7) at (2.2,3) [vertexY] {$v_7$};
\node (v8) at (3,7) [vertexY] {$v_8$}; \node (v9) at (3,5) [vertexY] {$v_9$};

\draw [line width=0.06cm] (v1) -- (v2);
\draw [line width=0.06cm] (v3) -- (v4);
\draw [line width=0.06cm] (v4) -- (v5);
\draw [line width=0.06cm] (v5) -- (v6);
\draw [line width=0.06cm] (v4) -- (v7);
\draw [line width=0.06cm] (v8) -- (v9);


\node [scale=1.2] at (2,-0.5) {$G_1$};
\node [scale=0.9] at (2,11.7) {\mbox{ }};
\end{tikzpicture}
&
\tikzstyle{vertexDOT}=[scale=0.25,circle,draw,fill]
\tikzstyle{vertexY}=[circle,draw, top color=gray!5, bottom color=gray!30, minimum size=11pt, scale=0.8, inner sep=0.99pt]
\tikzstyle{vertexZ}=[circle,draw, top color=gray!5, bottom color=gray!30, minimum size=11pt, scale=0.5, inner sep=0.49pt]
\tikzstyle{vertexW}=[circle,draw, top color=gray!5, bottom color=gray!30, minimum size=11pt, scale=1.0, inner sep=1.1pt]

\begin{tikzpicture}[scale=0.45]
\node (v1) at (2,11) [vertexY] {$v_1$};
\node (v2) at (2,9) [vertexY] {$v_2$};
\node (v3) at (1,7) [vertexY] {$v_3$};
\node (v4) at (1,5) [vertexY] {$v_4$};
\node (v5) at (1,3) [vertexY] {$v_5$};
\node (v6) at (1,1) [vertexY] {$v_6$};
\node (v7) at (2.2,3) [vertexY] {$v_7$};
\node (v8) at (3,7) [vertexY] {$v_8$}; \node (v9) at (3,5) [vertexY] {$v_9$};

\draw [line width=0.06cm] (v1) -- (v2);
\draw [line width=0.06cm] (v2) -- (v3);
\draw [line width=0.06cm] (v4) -- (v5);
\draw [line width=0.06cm] (v5) -- (v6);
\draw [line width=0.06cm] (v4) -- (v7);
\draw [line width=0.06cm] (v2) -- (v8);


\node [scale=1.2] at (2,-0.5) {$G_2$};
\node [scale=0.9] at (2,11.7) {\mbox{ }};
\end{tikzpicture}
& 
\tikzstyle{vertexDOT}=[scale=0.25,circle,draw,fill]
\tikzstyle{vertexY}=[circle,draw, top color=gray!5, bottom color=gray!30, minimum size=11pt, scale=0.8, inner sep=0.99pt]
\tikzstyle{vertexZ}=[circle,draw, top color=gray!5, bottom color=gray!30, minimum size=11pt, scale=0.5, inner sep=0.49pt]
\tikzstyle{vertexW}=[circle,draw, top color=gray!5, bottom color=gray!30, minimum size=11pt, scale=1.0, inner sep=1.1pt]

\begin{tikzpicture}[scale=0.45]
\node (v1) at (2,11) [vertexY] {$v_1$};
\node (v2) at (2,9) [vertexY] {$v_2$};
\node (v3) at (1,7) [vertexY] {$v_3$};
\node (v4) at (1,5) [vertexY] {$v_4$};
\node (v5) at (1,3) [vertexY] {$v_5$};
\node (v6) at (1,1) [vertexY] {$v_6$};
\node (v7) at (2.2,3) [vertexY] {$v_7$};
\node (v8) at (3,7) [vertexY] {$v_8$}; \node (v9) at (3,5) [vertexY] {$v_9$};

\draw [line width=0.06cm] (v1) -- (v2);
\draw [line width=0.06cm] (v2) -- (v3);
\draw [line width=0.06cm] (v3) -- (v4);
\draw [line width=0.06cm] (v5) -- (v6);
\draw [line width=0.06cm] (v2) -- (v8);
\draw [line width=0.06cm] (v8) -- (v9);

\draw [line width=0.01cm] (v1) to [out=300, in=50]  (v9);

\node [scale=1.2] at (2,-0.5) {$G_3$};
\node [scale=0.9] at (2,11.7) {\mbox{ }};
\end{tikzpicture}
\\ \hline
\end{tabular}
\end{center}
\caption{We illustrate the proof of Theorem~\ref{girth} with the example shown in (a), where we see a DFS tree (thick edges) of a graph $G$.
The remaining graphs depict $G_0$, $G_1$, $G_2$ and $G_3$, respectively, where $k=4$.}
\label{Thm41fig}
\end{figure}
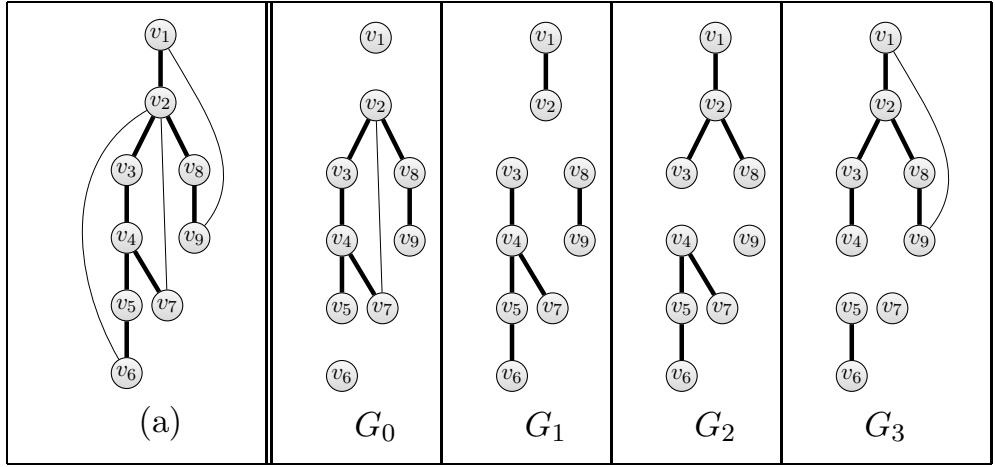

As every edge of $T$ belongs to $k-1$ of the $k$ subgraphs $G_j$'s, we note that summing the equations ${\rm mac}(G)\ge (w(G)+w(G_j))/2$ for all
$j=0,1,\ldots,k-1$ gives us the following:
\[
k  \cdot {\rm mac}(G)\ge k \; \frac{w(G)}{2} + \sum_{j=0}^{k-1} \frac{w(G_j)}{2} = k \; \frac{w(G)}{2} + (k-1) \frac{w(T)}{2} 
\]
Dividing the above inequality by $k$ gives us the desired bound.
\end{proof}


{To see that the lower bound of Theorem~\ref{girth} is tight, consider}
the unweighted cycle $C_{k+1},$ where $k$ is even. Then ${\rm mac}(C_{k+1})=k$ {and
the lower bound of Theorem \ref{girth} equals
\[
\frac{k+1}{2} + \frac{k-1}{2k} k = k.
\]
} 

Note that in Theorem \ref{girth}, $k$ is assumed to be even. When the girth $g$ is odd  we can 
use Theorem~\ref{girth} with $k=g-1$. This implies the following:

\begin{corollary}\label{girthODD}
Let $g$ be a positive odd integer. If the girth of $G$ is at least $g$, 
then ${\rm mac}(G)\ge \frac{w(G)}{2} + \frac{g-2}{2(g-1)} w(D)$, for every DFS tree $D$ in $G$.
\end{corollary}

{Corollary~\ref{girthODD} is tight for the cycle $C_g$ as the lower bounds for $C_g$ and ${\rm mac}(C_g)$ are both equal to $g-1$ in this case.}

\2

Recall that a graph of girth at least 4 is called {\em triangle-free}.

By Theorem~\ref{girth} for a triangle-free graph $G$, we have ${\rm mac}(G)\ge \frac{w(G)}{2} + \frac{3}{8} w(D),$ 
where $D$ is a DFS tree. However, by Theorem \ref{NPdfs} finding a DFS tree of maximum weight is {\sf NP}-hard.

The lower bound in the next theorem is of interest as it 
 implies that for a maximum weight spanning tree $T_{\max}$ of a triangle-free graph of $G,$ ${\rm mac}(G)\ge w(G)/2 + w(T_{\max})/4,$ 
which is stronger than the Poljak-Turz{\'{\i}}k lower bound. 

\begin{theorem}\label{tfree}
Let $G$ be a triangle-free graph and let $T$ be a spanning tree of $G$. Then ${\rm mac}(G)\ge w(G)/2 + w(T)/4.$
\end{theorem}
\begin{proof}
Let $T$ be rooted at vertex $u$.
Let $L_i$ the  set of vertices of $T$ at distance $i$ from $u$ in $D$.
Let $H_i$ be the subgraph of $T$ induced by the set of edges of $T$
between vertices $L_i$ and $L_{i+1}.$ {Since $G$ is triangle-free, the children of any node in $T$ form an independent set.
Thus, every $H_i$ is a disjoint union of stars.}
Let $G_0$ be the disjoint union of graphs $H_i$ with even $i$ and $G_1$ the disjoint union of graphs $H_i$ with odd $i.$
Thus, $G_0,G_1\in {\cal B}(G).$ Hence, by Theorem \ref{gen}, ${\rm mac}(G)\ge (w(G)+w(G_j))/2$ for $j=0,1.$ 
These bounds and $w(G_0)+w(G_1)=w(T)$ imply ${\rm mac}(G)\ge w(G)/2+w(T)/4.$
\end{proof}

Note that in Theorem~\ref{tfree}, $T$ can be any spanning tree. By Theorem~\ref{notANYtree} no similar bound holds if
we drop the condition that $G$ is triangle-free. {In fact, for every $\epX > 0$ there exist graphs $G$ with spanning tree $T$ such that ${\rm mac}(G) < w(G)/2 + \epX \cdot w(T)$.}

%

Define $\theta$ to be the largest value  such that ${\rm mac}(G) \ge \frac{w(G)}{2} + \theta \cdot w(T)$
holds for all spanning trees $T$ in a triangle-free graph $G$.

\begin{proposition}\label{boundsI}
$\frac{1}{4} \leq \theta \leq \frac{3}{8}$
\end{proposition}

\begin{proof}
Theorem~\ref{tfree} implies that $\theta \geq 1/4$.
Now consider the cycle $C_5$ with weight 1 on all edges. Then $w(C_5)=5$, $w(T)=4$ for all spanning trees $T$ and ${\rm mac}(C_5)=4$,
which implies that $\theta \leq 3/8$. 
\end{proof}

We think that determining the optimal value of $\theta$ is an interesting open problem.  In fact, we guess that $\theta=3/8$.

\begin{conjecture}\label{MAINconjI}
Let $G$ be triangle-free and let $T$ be a spanning tree of $G$. Then ${\rm mac}(G)\ge \frac{w(G)}{2} + \frac{3w(T)}{8}.$
\end{conjecture}

\section{ Triangle-free Graphs with Bounded Maximum Degree}\label{sec:conj}

Another interesting problem is to determine what happens to $\theta$ if we restrict ourselves to fixed maximum degrees.
That is, we let $\theta_{\Delta}$ be defined as the largest number for which the following holds:
If $G$ is triangle-free graph with $\Delta(G) \leq \Delta$ then ${\rm mac}(G)\ge \frac{w(G)}{2} + \theta_{\Delta} \cdot w(T)$.

It is not difficult to prove that $\theta_1=1/2$ (as if $\Delta(G) \leq 1$ then $G$ is bipartite) and
$\theta_2=3/8$ (due to $C_5$). This implies the following, by Proposition~\ref{boundsI}.

\begin{proposition}\label{boundsII}
$0.375 = \frac{3}{8} = \theta_2 \geq \theta_3 \geq \theta_4 \geq \ldots \geq \theta \geq \frac{1}{4} = 0.25$
\end{proposition}

In the rest of this section, we will first study triangle-free graphs with maximum degree at most 3 (Subsection \ref{sec:3}) and then those of arbitrary bounded maximum degree (Subsection \ref{sec:Delta}).

\subsection{Triangle-free subcubic graphs} \label{sec:3}

{A graph is {\em subcubic} if its maximum degree is at most 3.}
Triangle-free subcubic graphs have been widely studied, see e.g. \cite{BoLO86,HopkinsS82,Yannakakis78}.
 We will discuss such graphs next, before moving to the more general case of triangle-free graphs of bounded maximum degree.

The following result is well-known.

\begin{theorem}\label{knownI} \cite{BoLO86}
If $G$ is triangle-free subcubic graph then there exists a bipartite subgraph of $G$
containing at least $\frac{4}{5}|E(G)|$ edges.
\end{theorem}

We conjecture that the above theorem can be extended to the weighted case as follows.

\begin{conjecture}\label{MAINconjII}
Let $G$ be a triangle-free subcubic graph, then ${\rm mac}(G) \ge \frac{4}{5} w(G)$.
\end{conjecture}

Note that Theorem~\ref{knownI} implies that Conjecture~\ref{MAINconjII} holds in the case when all weights equal 1.
{As a support of Conjecture \ref{MAINconjII}, one can easily
show that if G is a  triangle-free subcubic graph, then ${\rm mac}(G) \ge \frac{2}{3} w(G)$, as follows.}
As $G$ is triangle-free with $\Delta(G) \leq 3$ it is known that $G$ is $3$-colorable, 
by Brook's Theorem.  Let $V_1,V_2,V_3$ be the three color classes in a proper $3$-coloring of $G$ and
assume that without loss of generality $w(V_1,V_2)$  is the maximum value in $\{w(V_1,V_2),w(V_1,V_3),w(V_2,V_3)\},$
where $w(V_i,V_j)$ is the total weight of all edges between $V_i$ and $V_j$. For each vertex $v \in V_3$ add it to 
$V_1$ if $w(v,V_2) \geq w(v,V_1)$ and otherwise add it to $V_2$. This results in a bipartition $(V_1,V_2)$ with 
weight at least $\frac{2}{3} w(G)$.

{In Section \ref{sec:proofs}, we will provide a proof of the following theorem, which approaches Conjecture~\ref{MAINconjII} even more.}

\begin{theorem}\label{mainProb}
Let $G$ be an edge-weighted triangle-free subcubic graph. 
Then ${\rm mac}(G) \geq \frac{8}{11} \cdot w(G)$.
\end{theorem}

\2

{Theorem~\ref{mainProb} will be used to prove the following theorem in Section \ref{sec:proofs}.}

\begin{theorem}\label{mainProbTree}
Let $G$ be an edge-weighted triangle-free subcubic graph
and let $T$ be an arbitrary spanning tree in $G$.
Then ${\rm mac}(G) \geq \frac{w(G)}{2} + 0.3193 \cdot w(T)$.
\end{theorem}

The above theorem implies the following corollary.

\begin{corollary}
$0.3193 \leq \theta_3 \leq \frac{3}{8} = 0.375$. 
\end{corollary}

Note that  
if Conjecture \ref{MAINconjI} holds for triangle-free subcubic graphs, then $\theta_3=3/8$.


\begin{proposition}\label{ImpliesC}
Conjecture~\ref{MAINconjII} implies Conjecture~\ref{MAINconjI} for triangle-free subcubic graphs.
\end{proposition}
\begin{proof}
Let $G$ be a triangle-free subcubic graph and let $T$ be any spanning tree in $G$.
As any tree is bipartite, we note that ${\rm mac}(G)\ge w(T).$ 
Conjecture~\ref{MAINconjII} would imply that ${\rm mac}(G)\ge \frac{4}{5} w(G)$.
Thus, we have the following, which completes the proof.
\[
{\rm mac}(G) \geq   \frac{3}{8} w(T) + \frac{5}{8} \left( \frac{4}{5} w(G) \right)   =   \frac{3}{8} w(T) + \frac{1}{2} w(G) 
\]
\end{proof}

The well-known pentagon conjecture of Erd{\"{o}}s \cite{Erdos84} states that a triangle-free graph $G$ on $n$ vertices has at most $(n/5)^5$
5-cycles. The bound is tight as every graph obtained from $C_5$ by replacing every vertex $x$ by a set of $n/5$ independent vertices (with the same adjacencies as in $C_5$) 
has exactly $(n/5)^5$ 5-cycles. This conjecture was proved independently in \cite{Grzesik12} and \cite{HatamiHKNR13}.

Let us consider another conjecture, which is on 5-cycles in triangle-free subcubic graphs.

\begin{conjecture}\label{5cycles}
Every triangle-free subcubic graph $G$ contains a set $E'$ of edges, such that every 5-cycle in $G$ contains exactly one edge from $E'.$
\end{conjecture}

Note that Conjecture~\ref{5cycles} holds for all {triangle-free subcubic} graphs, $G$, where every edge belongs to equally many 5-cycles, say $k$, by 
the following argument. By Theorem~\ref{knownI}, there exists a bipartite subgraph of $G$ with at least $4|E(G)|/5$ edges.
Let $E'$ be all the edges not in this subgraph.
In $G$ there are $k|E(G)|/5$ distinct $5$-cycles and each of them is covered
by (at least) one of the at most $|E(G)|/5$ edges in $E'$.
As each edge in $E'$ can cover at most $k$ $5$-cycles, we note that every $5$-cycle is covered exactly once,
which implies that Conjecture~\ref{5cycles} holds for $G$.  As a special case, 
Conjecture~\ref{5cycles} holds for the Petersen Graph, where every edge lies in four $5$-cycles. 
For example, in Figure~\ref{PetersenE}, the edge set $E'=\{x_1y_1,y_3y_4,x_3x_4\}$ intersects every $5$-cycle of the Petersen Graph  exactly once.

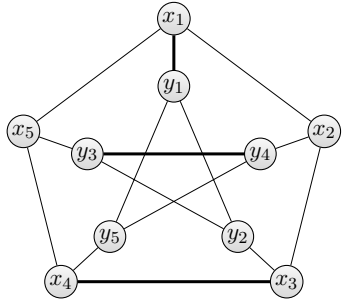
\begin{figure}[hbtp]
\begin{center}
\tikzstyle{vertexDOT}=[scale=0.25,circle,draw,fill]
\tikzstyle{vertexY}=[circle,draw, top color=gray!5, bottom color=gray!30, minimum size=11pt, scale=0.8, inner sep=0.99pt]
\begin{tikzpicture}[scale=0.5]
\node (x1) at (5,8) [vertexY] {$x_1$};
\node (x2) at (9,5) [vertexY] {$x_2$};
\node (x3) at (8,1) [vertexY] {$x_3$};
\node (x4) at (2,1) [vertexY] {$x_4$};
\node (x5) at (1,5) [vertexY] {$x_5$};

\node (y1) at (5,6.2) [vertexY] {$y_1$};
\node (y4) at (7.3,4.4) [vertexY] {$y_4$};
\node (y2) at (6.7,2.2) [vertexY] {$y_2$};
\node (y5) at (3.3,2.2) [vertexY] {$y_5$};
\node (y3) at (2.7,4.4) [vertexY] {$y_3$};

\draw [line width=0.01cm] (x1) -- (x2) -- (x3) -- (x4) -- (x5) -- (x1);
\draw [line width=0.01cm] (y1) -- (y2) -- (y3) -- (y4) -- (y5) -- (y1);

\draw [line width=0.04cm] (x1) -- (y1);
\draw [line width=0.04cm] (x3) -- (x4);
\draw [line width=0.04cm] (y3) -- (y4);
\draw [line width=0.01cm] (x2) -- (y4);
\draw [line width=0.01cm] (x3) -- (y2);
\draw [line width=0.01cm] (x4) -- (y5);
\draw [line width=0.01cm] (x5) -- (y3);

\end{tikzpicture}
\end{center}
\caption{Petersen graph with $E'$ in boldface}
\label{PetersenE}
\end{figure}

The following result is a link between Conjectures~\ref{MAINconjII} and \ref{5cycles}, which provided Conjecture~\ref{MAINconjII} holds, demonstrates that a lower bound on ${\rm mac}(G)$ can be used
to establish a structural result on unweighted graphs.

\begin{proposition}\label{Implies5cycles}
Conjecture~\ref{MAINconjII} implies Conjecture~\ref{5cycles}.
\end{proposition}
\begin{proof}
Let $G=(V,E)$ be a triangle-free graph with maximum degree 3.  For each $e\in E$, define $w(e)$ to be the number of 5-cycles that contain  $e$ in $G.$ 
If there are exactly $c$ distinct 5-cycles in $G,$ then  $w(G)=5c.$ Let $B=(V,F)$ be a maximum weight bipartite subgraph in $G$ and let $E' =E\setminus F.$  
Note that every 5-cycle in $G$ must contain an edge from $E'$, which implies that $w(E') \ge c.$   If Conjecture ~\ref{MAINconjII} holds,
 then ${\rm mac}(G)\ge 4w(G)/5,$ which implies that $w(E') \le w(G)/5 = c. $ 
Therefore, $w(E')=c,$ which implies that every 5-cycle contains exactly one edge from $E',$ and thus Conjecture \ref{5cycles} holds.
\end{proof}

\subsection{Triangle-free graphs with arbitrary bounded maximum degree} \label{sec:Delta}

Using Shearer's randomized algorithm \cite{Shearer92} we can obtain the following bound. 

\begin{theorem}\label{Shearer}
Let $G$ be a weighted triangle-free graph with $\Delta(G)\le \Delta$. 
Then ${\rm mac}(G) \geq  s_{\Delta}\cdot w(G),$ where $ s_{\Delta}=\frac{1}{2} + \frac{1}{4\sqrt{2\Delta}}.$
\end{theorem}
\begin{proof}
{Shearer's randomized algorithm takes $G$ as an input and constructs a random cut $C$ of $G$ as follows. In Step 1, we partition $V(G)$ randomly and uniformly into two sets $A$ and $B.$
Call a vertex $x$ in $G$ {\em good} if more than half of the vertices adjacent to it lie in the other set. If exactly half of the vertices adjacent to $x$ lie in the other set call $x$ good with probability 1/2.
Otherwise call $x$ {\em bad}. In Step 2, leave the good vertices where they are and redistribute the bad vertices randomly and uniformly into $A$ and $B.$ The edges between $A$ and $B$ induce the cut $C.$
} 

It is shown in the proof of Theorem 1 in \cite{Shearer92} that an edge $uv$ of $G$ is in $C$ with probability $\theta_{uv}=\frac{1}{2} + \frac{1}{4} (\rho_u - \frac{1}{2}) + \frac{1}{4}  (\rho_v - \frac{1}{2}),$
where $\rho_x$ is the probability that $x\in \{u,v\}$ is good after Step 2 provided that $u,v$ are put into different sets in Step 1. Let $f_x=\rho_x - \frac{1}{2}$, where $x\in \{u,v\}$. It follows from the proof of Theorem 1
and from Lemma 1 in \cite{Shearer92} that for $x\in \{u,v\}$
$$f_x\ge \frac{1} {2\sqrt{2\cdot {\rm deg}(x)}}\geq  \frac{1}{2\sqrt{2\Delta}}$$ 

Observe that the expected value of the weight of $C$ is
\begin{eqnarray*}
\sum_{uv \in E} w(uv) \theta_{uv} &= &\\
W(G)/2 + \sum_{uv \in E} w(uv)(f_u+f_v)/4 &\geq &\\
W(G)/2 + (1/4\sqrt{2\Delta}) W(G) &&
\end{eqnarray*}
Thus, the expected weight of $C$ is at least $ s_{\Delta}\cdot w(G)$, {where  $ s_{\Delta}=\frac{1}{2} + \frac{1}{4\sqrt{2\Delta}}$}. Hence, $G$ has a cut with at least such a weight. 
\end{proof}

Note that we can find a cut of $G$ of weight at least $ s_{\Delta}\cdot w(G)$ in polynomial time using the derandomization method of conditional probabilities \cite{AloSpe}.
Also note that for $\Delta(G)\le 3,$ the bound in Theorem \ref{Shearer} is as follows: ${\rm mac}(G) \geq 0.602\cdot w(G)$. However, Theorem \ref{mainProb} provides a significantly better bound: ${\rm mac}(G) \geq  0.727\cdot w(G).$ This indicates that Theorem  \ref{Shearer} can be improved at least for small values of $\Delta$ {and indeed Theorem \ref{th:Y} provides such an improvement Theorem  \ref{Shearer} for small values of $\Delta$.  The following lemma will be used to prove Theorems~\ref{mainProb} and \ref{th:Y}.}


\begin{lemma}\label{mainMatch}
Let $G$ be a weighted triangle-free graph with  $\Delta(G)\le \Delta$ and let $M$ be a matching in $G$.
Then ${\rm mac}(G) \geq \frac{\Delta}{2\Delta-1} (w(G)-w(M)) + w(M) $.
\end{lemma}

\begin{proof}
Let $G$ and $M$ be defined as in the statement of the theorem and $\Delta(G)\le \Delta$.
Let $G_1$ be the graph obtained  by contracting each edge $e_i \in M$ to a vertex $x_i$.
If $G_1$ has parallel edges, for every pair $u,v$ of vertices with parallel edges, delete 
all but one of the parallel edges and let the weight 
of the remaining edge between $u$ and $v$ be the sum of the original edges. We will denote the resulting graph by $G_2$.
Note that $w(G_2) = w(G) - w(M)$ and $G_2$ is a simple graph with $\Delta(G_2) \leq 2\Delta-2$.

By Vizing's theorem, $G_2$ has a proper $c$-edge-coloring, where $c\le 2\Delta-1.$
Let $M''_i$ denote the edges of color $i$ in such a proper $c$-edge-coloring ($i \in [c]$). Note that
$M''_i$ is a matching in $G_2$. Observe that $M''_i$ corresponds to an induced matching in $G$ (via $G_1$) denoted by $M_i.$
{Note that the parallel edges do not create a problem with the matchings since $G$ is triangle-free. Every component $C$ in $G[M \cup M_i]$ has  
at most four vertices (otherwise, contraction of edges of $M$ in $C$ would create a pair on non-parallel edges of $M_i$ in $G_1$, which is impossible)}. Hence,
 $G[M \cup M_i] \in {\cal B}(G)$ (as $G$ is triangle-free). Thus, Theorem~\ref{gen}  implies that
$
{\rm mac}(G) \ge \frac{w(G)+ w(M)+w(M_i)}{2}.
$
Summing the above over $i=1,2,\ldots,c$ and 
noting that every edge not in $M$ belongs to one of the matching $M_i$, we have $$c \cdot {\rm mac}(G) \ge \frac{c}{2}(w(G)+w(M))+\frac{w(G)-w(M)}{2}.$$
Hence, 
\[
\begin{array}{rcl} 
{\rm mac}(G) & \ge & \frac{w(G)+w(M)}{2}+\frac{w(G)-w(M)}{2c}\\
& \ge & \frac{w(G)+w(M)}{2}+\frac{w(G)-w(M)}{2(2\Delta-1)}\\
 & = & \frac{\Delta}{2\Delta-1} (w(G)-w(M)) + w(M) \\
\end{array}
\]
\end{proof}

\begin{theorem}\label{th:Y}
Let $G$ be a weighted triangle-free graph with  $\Delta(G)\le \Delta.$ Then
${\rm mac}(G) \geq t_{\Delta}\cdot w(G),$ where $t_{\Delta}=1/2 + (3\Delta -1)/(4\Delta^2+2\Delta-2).$
\end{theorem}
\begin{proof}
By Vizing's theorem, $G$ has a $(\Delta + 1)$-coloring. Let $M_i$ be the edges with color $i \in [\Delta +1],$ and note that $M_i$ is a matching. Using Lemma \ref{mainMatch} for each $M_i$,
we obtain ${\rm mac}(G) \geq \frac{\Delta}{2\Delta-1} (w(G)-w(M_i)) + w(M_i) $ for every $i\in [\Delta + 1].$ 
By summing up these inequalities and simplifying the resulting inequality (by using the fact that 
$w(M_1)+w(M_2)+\cdots +w(M_{\Delta+1})=w(G)$), we obtain the following:
\[
\begin{array}{rcl} \2
{\rm mac}(G) & \geq & \frac{1}{\Delta +1} \sum_{i=1}^{\Delta+1} \left( \frac{\Delta}{2\Delta-1} (w(G)-w(M_i)) + w(M_i) \right) \\ \2
 & = &  \frac{1}{\Delta +1} \sum_{i=1}^{\Delta+1} \left( \frac{\Delta}{2\Delta-1} w(G) + \frac{\Delta -1}{2\Delta-1} w(M_i) \right) \\ \2
 & = & \frac{\Delta}{2\Delta-1} w(G) +  \frac{1}{\Delta +1} \times \frac{\Delta -1}{2\Delta-1} w(G) \\ \2
 & = & \left( \frac{1}{2} + \frac{3\Delta-1}{4\Delta^2+2\Delta-2} \right) w(G)  \\ 
\end{array}
\]
\end{proof}

Let us compare the bounds of Theorems \ref{Shearer} and \ref{th:Y}. We have $t_{\Delta}>s_{\Delta}$ if and only if $\Delta \le 16.$
In fact, a selected number of values of $s_{\Delta}$ and $t_{\Delta}$ can be seen below.
\begin{center}
\begin{tabular}{|c||c|c|}\hline
$\Delta$ & $s_{\Delta}$ & $t_{\Delta}$ \\ \hline \hline
 1 & 0.6768 & 1.0000 \\ \hline 
 2 & 0.6250 & 0.7778 \\ \hline 
 3 & 0.6021 & 0.7000 \\ \hline 
 4 & 0.5884 & 0.6571 \\ \hline 
$\cdots$ & $\cdots$ & $\cdots$ \\ \hline
16 & 0.5442 & 0.5446 \\ \hline 
17 & 0.5429 & 0.5421 \\ \hline 
\end{tabular}
\end{center}

Since $t_3=0.7,$ we have ${\rm mac(}G) \geq 0.7 w(G)$ for a triangle-free $G$ with $\Delta(G)\le 3.$ The gap between 0.7 and the coefficient 0.8 of Conjecture \ref{MAINconjI} is just 0.1, but it does not seem to be easy to bridge this gap as the proof of Theorem \ref{mainProb} in the next section shows.

Lemma~\ref{mainMatch} seems interesting in its own right and gives rise to the following open problem.

\begin{openProblem}\label{openProbMatch}
For each $\Delta \geq 1$ determine the maximum value, $c_{\Delta}$, such that for every
edge-weighted triangle-free graph $G$ with maximum degree at most  $\Delta$ and matching $M$ in $G,$ the following holds.
\[
{\rm mac}(G) \geq c_{\Delta} (w(G)-w(M)) + w(M) 
\]
\end{openProblem}

The following proposition determines  $c_1$, $c_2$ and $c_3$ precisely.

\begin{proposition}\label{mainMatchTight}
$c_1 = 1$, $c_2=\frac{2}{3}$ and $c_3 = 0.6$.
\end{proposition}
\begin{proof}
By Lemma~\ref{mainMatch} we note that $c_1 \geq 1$, $c_2 \geq \frac{2}{3}$ and $c_3 \geq 0.6$.
Clearly $c_1 \leq 1$ (consider unweighted $K_2$ and $M=\emptyset$), which implies that $c_1=1$. 

If $G$ is an unweighted $5$-cycle and $M$ is a matching
of size two in $G$, then ${\rm mac}(G) = 4 = \frac{2}{3} (|E(G)|-|E(M)|) + |E(M)|$, 
which implies that  $c_2 \leq \frac{2}{3}$. Therefore $c_2 = \frac{2}{3}$.

Let $G$ be the Petersen graph depicted in Figure~\ref{PetersenG}. That is, the edge set consists of the edges of
two $5$-cycles, $C_x=x_1 x_2 x_3 x_4 x_5 x_1$ and $C_y=y_1 y_2 y_3 y_4 y_5 y_1$ and the matching $M=\{ x_1 y_1, x_2 y_4, x_3 y_2, x_4 y_5, x_5 y_3 \}$.
Let the weight of the edges in $M$ be 10 and the weight of all other edges in $G$ be 1.  Then $w(M)=50$ and $w(G)-w(M)=10$.
This implies that any maximum weight cut, $(A,B)$, must include all edges on $M$.  As 
$C_x$ and $C_y$ are both $5$-cycles at most $4$ edges from each can belong to $(A,B)$.
If four edges from $C_x$ belongs to $(A,B)$ then we note that at most 2 edges from $C_y$ belong to $(A,B)$ (if all edges in $M$ belong to $(A,B)$) 
and analogously
if four edges from $C_y$ belongs to $(A,B)$ then we note that at most 2 edges from $C_x$ belong to $(A,B)$. This implies that 
${\rm mac}(G) \leq w(M)+6 = w(M) + 0.6 (w(G)-w(M))$. Therefore $c_3 \leq 0.6$, which implies that $c_3=0.6$.
\end{proof}

\begin{figure}[hbtp]
\begin{center}
\tikzstyle{vertexDOT}=[scale=0.25,circle,draw,fill]
\tikzstyle{vertexY}=[circle,draw, top color=gray!5, bottom color=gray!30, minimum size=11pt, scale=0.8, inner sep=0.99pt]
\begin{tikzpicture}[scale=0.5]
\node (x1) at (5,8) [vertexY] {$x_1$};
\node (x2) at (9,5) [vertexY] {$x_2$};
\node (x3) at (8,1) [vertexY] {$x_3$};
\node (x4) at (2,1) [vertexY] {$x_4$};
\node (x5) at (1,5) [vertexY] {$x_5$};

\node (y1) at (5,6.2) [vertexY] {$y_1$};
\node (y4) at (7.3,4.4) [vertexY] {$y_4$};
\node (y2) at (6.7,2.2) [vertexY] {$y_2$};
\node (y5) at (3.3,2.2) [vertexY] {$y_5$};
\node (y3) at (2.7,4.4) [vertexY] {$y_3$};

\draw [line width=0.01cm] (x1) -- (x2) -- (x3) -- (x4) -- (x5) -- (x1);
\draw [line width=0.01cm] (y1) -- (y2) -- (y3) -- (y4) -- (y5) -- (y1);

\draw [line width=0.04cm] (x1) -- (y1);
\draw [line width=0.04cm] (x2) -- (y4);
\draw [line width=0.04cm] (x3) -- (y2);
\draw [line width=0.04cm] (x4) -- (y5);
\draw [line width=0.04cm] (x5) -- (y3);

\end{tikzpicture}
\end{center}
\caption{Petersen graph $G$ and matching $M$ in boldface}
\label{PetersenG}
\end{figure}
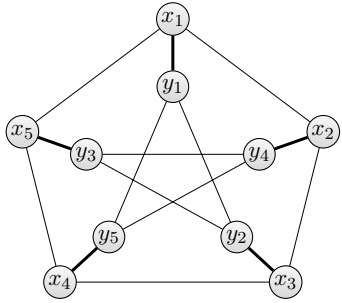

Proposition~\ref{match} implies that $c_{\Delta} \geq 0.5$ for all $\Delta \geq 1$, which by Theorem~\ref{mainMatchTight} implies
that the following holds.
\[
0.6 = c_3 \geq c_4 \geq c_5 \geq \cdots \geq 0.5
\]
We will finish this section with the following asymptotic result for $c_{\Delta}.$

\begin{proposition}
We have $\lim_{\Delta \rightarrow \infty} c_{\Delta} = 0.5.$
\end{proposition}
\begin{proof}
Since the sequence is monotonically decreasing and bounded by 0.5, the limit exists and it is at least 0.5. Suppose the limit $L$ is larger than 0.5.

Alon \cite{Alon96} proved that there exists some absolute positive constant $c'$ so that for every $m$ there exists a triangle-free graph $G$ with $m$ 
edges for which no bipartite subgraph has $m/2+c'\cdot m^{4/5}$ edges.
Let all edges of $G$ be of weight 1 and let $M$ be a matching of $G.$   If $|E(G)|$ is large enough for $c'm^{4/5} < (L-0.5)m$ to hold, we have
\[
\begin{array}{rcl} 
{\rm mac}(G) & \le & m/2 + c'm^{4/5}\\
& < & Lm \\
& \leq & Lm + (1-L)|M| \\
 & = & L(m-|M|) + |M|  \\
 & \le & c_{\Delta(G)}(m-|M|) + |M| \\
\end{array}
\]
This contradiction completes the proof.
\end{proof}


\section{Proofs of Theorems~\ref{mainProb} and \ref{mainProbTree}}\label{sec:proofs}

We first need the following lemmas.

\begin{lemma}\label{girth2}
Let $k$ be any positive integer and let $T$ be any spanning tree in a graph $G$ and let $e^* \in E(T)$ be arbitrary.
If for all $e \in E(G) \setminus E(T)$ we have that $T+e$ contains no odd cycle of length $2k-1$ or less, then 
${\rm mac}(G)\ge \frac{w(G)}{2} + \frac{k-1}{2k} w(T) + \frac{1}{2k} w(e^*)$.
\end{lemma}
\begin{proof}
Let $G$, $T$, $e^*$ and $k$ be defined as in the lemma.
Let $e^* = r_1 r_2$ and let $L_i$ contain the vertex $u$ if and only if $i=\min\{d_T(r_1,u),d_T(r_2,u)\}$,
where $d_T(r_i,u)$ denotes the distance from $r_i$ to $u$ in $T$.
Note that $L_0 = \{r_1,r_2\}$.

For $j=0,1,\ldots,k-1$, let $G_j$ be the subgraph induced by the edges of $T$ minus those between $L_i$ and $L_{i+1}$ for
all $i = j$ (mod $k$). 

Let $T^*$ be an arbitrary component in $G_j$ (for any $j\in \{0,1,\ldots,k-1\}$). As $T^*$ is a subgraph of $T$ we note that 
$T^*$ is a tree. 
Let $e \in E(G)\setminus E(T)$ have both end-points in $V(T^*)$. We will now show that $T^* + e$ contains no odd-cycle.
Let $C$ be the unique cycle in $T^* + e$. By the construction of $G_j$ we note that the path $C-e$ has length at most $2k-1$ 
(in fact, if $C$ does not contain $e^*,$ then the length is at most $2k-2$)
and that the cycle $C$ therefore has length at most $2k$.
As $T+e$ contains no odd cycle of length $2k-1$ or less (and $2k$ is even) we note that $T^* + e$ contains no odd-cycle.

Therefore $G[V(T^*)]$ is bipartite,
which implies that $G[V(T^*)]$ belongs to ${\cal B}(G)$ and and by Theorem~\ref{gen}, 
${\rm mac}(G)\ge (w(G)+w(G_j))/2$ for every $j=0,1,\ldots,k-1.$ 

As $e^*$ belongs to all the $k$ subgraphs $G_j$ and every edge 
of $T-e^*$ belongs to $k-1$ of the $k$ subgraphs $G_j$, we note that summing the equations 
${\rm mac}(G)\ge (w(G)+w(G_j))/2$ for all $j=0,1,\ldots,k-1$ gives us the following: 
\[
k  \cdot {\rm mac}(G)\ge k \; \frac{w(G)}{2} + \sum_{j=0}^{k-1} \frac{w(G_j)}{2} = k \; \frac{w(G)}{2} + (k-1) \frac{w(T)}{2} + \frac{w(e^*)}{2}
\]
Dividing the above inequality by $k$ gives us the desired bound.
\end{proof}

\begin{lemma}\label{lemmaProb}
Let $G$ be a subcubic edge-weighted graph and let $T$ be a spanning tree in $G$.
Let $r$ be the length of a shortest odd cycle in $T+e$ for any edge $e \in E(G) \setminus E(T)$.
Let $p$ be arbitrary such that $0 \leq p \leq 1$. Then the following bound holds.
\[
{\rm mac}(G)\ge \frac{p+1}{2} w(T) + \frac{1-p^{r-1}}{2} (w(G)-w(T)).
\]
\end{lemma}
\begin{proof}
  Let $G$, $T$, $r$ and $p$ be defined as in the lemma.  Pick each edge of $T$ with probability $p$ and denote the
resulting set of edges by $E^*$. Let $T_1^*,T_2^*,\ldots,T_l^*$ be the connected components of the graph 
$G^* = (V(G),E^*)$. Note that all $T_i^*$ are subtrees of $T$ (some of which may contain only one vertex). 

Let $(X_i,Y_i)$ be a bipartition of $T^*_i$ for all $i=1,2,\ldots,l.$ 
For each $i\in [l],$ randomly and uniformly assign $X_i$ color 1 or 2 and $Y_i$ the opposite color.
Let $A$ be all vertices of color 1 and let $B$ be all vertices of color 2. 
Now every edge in $E^*$ lies in the cut induced by $(A,B)$.
Consider an arbitrary edge $e \in E(G) \setminus E^*$ and let $e=uv$. 
Assume that $u \in V(T^*_i)$ and $v \in V(T^*_j)$.
We will show the following claims.

\2

{\bf Claim 1:} {\em If $e \in E(T) \setminus E^*$, then $e$ lies in the cut induced by $(A,B)$ with probability $\frac{1}{2}$ (given $E^*$). }

\2

{\em Proof of Claim~1:} Note that $i \not= j$, as adding the edge $e$ to $G^*$ does not create a cycle (all edges belong to the tree
$T$).  Therefore $u$ and $v$ will be assigned colors 1 and 2 randomly and independently. This completes the proof of Claim~1.

\2

{\bf Claim 2:} {\em If $e \in E(G) \setminus E(T)$, then $e$ lies in the cut induced by $(A,B)$ with probability
at least $\frac{1-p^{r-1}}{2}$.}

\2

{\em Proof of Claim~2:} Let $C_e$ denote the unique cycle in $T+e$.

First consider the case when $i=j$ and $|E(C_e)|$ is even.  
In this case $e$ lies in the cut induced by $(A,B)$ with probability 1 (given $E^*$).

Now consider the case when $i=j$ and $|E(C_e)|$ is odd. By the definition of $r$ in the statement of the lemma
we must have $|E(C_e)| \geq r$. Furthermore all edges in $C_e - e$ belong to $E^*$. So this case only happens with 
probability at most $p^{r-1}$ (as each edge in $C_e - e$ has probability $p$ of belonging to $E^*$).

Finally, if $i \not=j$ then $e$ lies in the cut induced by $(A,B)$ with probability $\frac{1}{2}$ (given $E^*$).

So with probability at least $1-p^{r-1}$, we are in a case where $e$ will belong to the cut induced by $(A,B)$ with probability at least 1/2. 
Therefore $e$ will belong to the cut induced by $(A,B)$ with probability at least $\frac{1}{2} \times (1-p^{r-1})$.
This completes the proof of Claim~2.

\2

We now return to the proof of Lemma~\ref{lemmaProb}.
If $e \in E(T)$ then by Claim~1 and the law of total probability, the probability that $e$ belongs to the cut induced by $(A,B)$ is 
$p + (1-p)\frac{1}{2} = \frac{1+p}{2}$.
By Claim~2, if $e \in E(G) \setminus E(T)$ then 
$e$ lies in the cut induced by $(A,B)$ with probability at least $\frac{1-p^{r-1}}{2}$.
This completes the proof.
\end{proof}

Recall the statement of Theorem~\ref{mainProb}.

\2

\noindent
{\bf Theorem~\ref{mainProb}.} {\em Let $G$ be an edge-weighted triangle-free graph with $\Delta(G) \leq 3$.
Then ${\rm mac}(G) \geq \frac{8}{11}  w(G)$.
}

\begin{proof}
We will first show that we may restrict our attention to $3$-regular triangle-free graphs.
Let $G'$ be equal to $K_{3,3}$ where one edge $uv$ has been subdivided. That is, $uv$ has been replaced by 
a path $uwv$. For every vertex, $s \in G$ we add $3-d_G(s)$ copies of $G'$ to $G$ and add an edge from each $w$-vertex in
the $G'$s to $s$. Note that the resulting graph is $3$-regular and triangle-free.
Furthermore giving all the new edges a weight of zero, shows that if the theorem holds for this new graph then 
it also holds for $G$. We may therefore without loss of generality assume that $G$ is $3$-regular.

By Brook's Theorem we note that $\chi(G) \leq 3$. Let $c:\ V(G)\rightarrow \{1,2,3\}$ be a proper 3-coloring 
of $G$ {and let $V_1,V_2,V_3$ be the color classes of $c$.}
For a given vertex $v \in V (G),$ if a color appears exactly once in $N(v)$ {(i.e. one vertex of $N(v)$ has a color $i$ and the other two vertices of $N(v)$ have a color $j\ne i$),}
we let $s(v)$ be the
neighbor of $v$ with that color. Otherwise, $s(v)$ is not defined.

Define the digraph $D^*$ such that $V(D^*)=V(G)$  and the arc set of $D^*$ is as follows:
\[
A(D^*) = \{ v s(v) \; | \; v\in V(G) \; \& \; \mbox{$s(v)$ is defined} \}
\]

{Let $\Delta^+(D^*)$ denote  the maximum out-degree of $D^*$.} Note that $\Delta^+(D^*) \leq 1$ by the construction of $D^*$ and that $D^*$ may contain $2$-cycles. If $uvu$ is a 
$2$-cycle in $D^*$ then $s(u)=v$ and $s(v)=u$. Let $G^* = UG(D^*)$ i.e. $G^*$ is the underlying graph of $D^*$
and contains all edges of the form $vs(v)$, where $v \in V(G)$. 

For every edge $e \in E(G),$ let $V^*(e) = \{ v \; | \; vs(v)=e \}$ and let
$A_i = \{ e \; | \; |V^*(e)|=i \}$ for $i=0,1,2$.
That is, $A_0$ contains all edges $e$ that are not of the form $v s(v)$ for any $v,$ 
$A_2$ contains all edges $uv$ where $s(u)=v$ and $s(v)=u$ and $A_1 = E(G) \setminus (A_0 \cup A_2)$.
Also note that $E(G^*) = A_1 \cup A_2$.
We will now prove the following claims.

\2

{\bf Claim A:} {\em  ${\rm mac}(G) \ge w(A_0) + \frac{2}{3} w(A_1) + \frac{1}{3} w(A_2)$.}

\2

{\em Proof of Claim~A:} For all $i\in \{1,2,3\}$, we define
\[
C_i = E(G) \setminus \{ v s(v) \; | \;  v \in V_i \; \& \; \mbox{$s(v)$ is defined} \}.
\]
Note that $C_i$ induces a bipartite graph as every vertex in $V_i$ only has edges to one of the two other sets.
So ${\rm mac}(G) \geq w(C_i)$ for $i=1,2,3$. Note that 
every edge in $A_i$ appears in $3-i$ of the sets $C_1,C_2,C_3$, which implies the following:
\[
{\rm mac}(G) \geq \frac{1}{3} ( w(C_1) + w(C_2) + w(C_3) ) = w(A_0) + \frac{2}{3} w(A_1) + \frac{1}{3} w(A_2).
\]
This completes the proof of Claim~A.

\2

{\bf Claim B:} {\em  If $p_1p_2 \in E(G)$ and $p_2p_3 \in A(D^*)$, where $p_1 \not= p_3$, then 
$\{c(p_1),c(p_2),c(p_3)\} = \{1,2,3\}$.  This implies that $c(p_1) = 6 - c(p_2) - c(p_3)$.}

{Furthermore,  if $R=r_1,r_2,r_3, \ldots, r_s$ is a directed path in $D$,
then  $c({r_{1}})=c({r_{4}})=c({r_{7}})= \ldots$ and $c({r_{2}})=c({r_{5}})=c({r_{8}})= \ldots$ and $c({r_{3}})=c({r_{6}})=c({r_{9}})= \ldots$.
}

\2

{\em Proof of Claim~B:}  As all edges in $G$, and therefore also arcs in $D^*$, go between different $V_i$-sets, we note that
$c({p_1})\not= c({p_2})$ and $c({p_2}) \not= c({p_3})$. 
As $p_2p_3 \in A(D^*)$ we have $s(p_2) = p_3$, which implies that there is only one edge 
from $p_2$ to $V_{c({p_3})}$. Furthermore this edge is $p_2p_3$. 
Therefore $p_1 \not\in V_{c({p_3})}$ as otherwise $p_2$ would have two edges to $V_{c({p_3})}$.
So, $c({p_1})\not= c({p_3})$. This implies that $c({p_1})$, $c({p_2})$ and $c({p_3})$ take on three distinct values, 
which completes the {proof of the first part of Claim~B.
The second part of Claim~B follows immediately from the first part.}

\2

{\bf Claim C:} {\em Let $e \in A_0$ and assume that $C$ is a cycle in $G^* +e$ containing $e$.
Then $|E(C)| = 0$ (mod 3). } 

\2

{
{\em Proof of Claim~C:} Let $C$ be a cycle in $G^* +e$ containing $e$, where $ e\in A_0$.
Let $P$ be the path $C-e$ and assume $P=p_1p_2p_3 \ldots p_l$. Note that $e= p_1 p_l$. 
Assume without loss of generality that $c(p_1)=1$ and $c({p_l})=2$.
We now consider the following two cases. 

\2

{\em Case C.1. $p_1 p_2 p_3 \ldots p_l$ or $p_l p_{l-1} p_{l-2} \ldots p_1$ is a directed path in $D^*$.}  

Assume without loss of generality that $p_1 p_2 p_3 \ldots p_l$ is a directed path in $D^*$. 
As $c({p_1})=1$ and $c({p_l})=2$, we note that $c({p_2})=3$ and $c({p_3})=2$, by Claim~B. 
Therefore, as $c({p_l})=2=c({p_3})$ we must have that $l$ is divisible by 3, by Claim~B, which completes the proof of Case~C.1.

\2

{\em Case C.2. $p_1 p_2 p_3 \ldots p_l$ and $p_l p_{l-1} p_{l-2} \ldots p_1$ are not directed paths in $D^*$.}

As $\Delta^+(D^*) \leq 1$ this implies that
there exists a $p_i \in \{p_2,p_3,\ldots,p_{l-1}\}$ such that $p_i$ has two in-neighbors in $P.$
Let $P'=p_1 p_2 \ldots p_i$ and let $P''=p_l p_{l-1} p_{l-2} \ldots p_i$ and note that $P'$ and $P''$ are both directed paths in $D^*$.
As $c({p_1})=1$ and $c({p_l})=2$ we note that, by Claim~B, $c({p_{l-1}})=c({p_2})=3$. 

If $c({p_i})=1$ then $|E(P')| = 0$ (mod 3) and $|E(P'')| = 2$ (mod 3), implying that $|E(C)| = 0$ (mod 3) as desired.
If $c({p_i})=2$ then $|E(P')| = 2$ (mod 3) and $|E(P'')| = 0$ (mod 3), implying that $|E(C)| = 0$ (mod 3) as desired.
And finally if $c({p_i})=3$ then $|E(P')| = 1$ (mod 3) and $|E(P'')| = 1$ (mod 3) again implying that $|E(C)| = 0$ (mod 3) as desired.
This completes the proof for Case~C.2.
}


\2

{\bf Claim D:} {\em If $C$ is a cycle in $G^*$ then $|E(C)| = 0$ (mod 3) and $C$ contains no chord in $G$. }

\2

{\em Proof of Claim~D:} Let $C = p_1 p_2 p_3 \ldots p_l p_1$ be a cycle of length $l$ in $G^*$.
As $\Delta^+(D^*) \leq 1$ we note that $C$ is a directed cycle in $D^*$. Without loss of generality 
assume that $c({p_1})=1$ and $c({p_2})=2$ (otherwise rename the $V_i$'s). By Claim~B we note that 
$c({p_l})=3$ (due to the path $p_l p_1 p_2$) and $c({p_3})=3$ (due to the path $p_1 p_2 p_3$).
Continuing using Claim~B we note that the following holds.
\begin{eqnarray*}
c({p_l})&=&3 \mbox{, }  c({p_1})=1 \mbox{, } c({p_2})=2 \mbox{, } c({p_3})=3 \mbox{, }\\
c({p_4})&=&1 \mbox{, }  c({p_5})=2 \mbox{, } c({p_6})=3 \mbox{, } c({p_7})=1 \mbox{, } \ldots 
\end{eqnarray*}
We see that $c({p_j})=3$ if and only if $j$ is divisible by three, which implies 
that $l$ is divisible by three (as $c({p_l})=3$).
Therefore  $|E(C)| = 0$ (mod 3).

For the sake of contradiction assume that $C$ has a chord, $p_i p_j$ in $G$, where $i<j$.
Consider the two cycles 
$$C_1 = p_i p_{i+1} \ldots p_j p_i \mbox{ and }
C_2 = p_j p_{j+1} \ldots p_l p_1 p_2 \ldots p_i p_j$$
By Claim~C we note that $|E(C_1)|$ and $|E(C_2)|$ are both divisible by three.
However $|E(C_1)|+|E(C_2)| = |E(C)|+2$ (as the edge $p_i p_j$ is counted twice).
This is a contradiction as $|E(C_1)|+|E(C_2)|$ is divisible by three, but $|E(C)|+2$ is not 
(as $|E(C)|$ is divisible by three, by our above arguments),
which completes the proof of Claim~D.

\2





{\bf Claim E:} {\em If $P$ is a $(u,x)$-path in $D^*,$ $Q$ is a $(v,x)$-path in $D^*$ and there exists
an arc $xy \in D^*$, 
such that $|A(P)|,|A(Q)| \geq 1$ and $yx$ is not an arc on either $P$ or $Q$}, then $uv \not\in E(G)$.

\2

{\em Proof of Claim~E:} Let $P=p_1p_2 \ldots p_a$  ($a \geq 2$) and $Q=q_1 q_2  \ldots q_b$ ($b \geq 2$), 
where $p_1=u$, $q_1=v$ and $p_a=q_b=x$ and let $xy \in A(D^*)$ be defined as in the statement of Claim~E. 

{
Without loss of generality assume that $c({x} )=2$ and $c(y)=3$. 
Using Claim~B on the paths $Py$ and $Qy$ we note that 
all arcs in $P$ and in $Q$ go from $V_1$ to $V_2$ or from $V_2$ to $V_3$ or from $V_3$ to $V_1$ in $D^*$.
}

Assume for the sake of contradiction that $uv \in E(G)$ (i.e. $p_1q_1 \in E(G)$). As $p_1p_2,q_1q_2 \in A(D^*)$ and $p_1q_1 \in E(G)$,
Claim~B implies the following:
\[
c({p_2} )= 6 - c({p_1}) - c({q_1}) = c({q_2})
\]
As all arcs in $P$ and in $Q$ go from $V_1$ to $V_2$ or from $V_2$ to $V_3$ or from $V_3$ to $V_1$ in $D^*$ this implies that
$c({p_1}) = c({q_1})$ (as $c({p_2}) = c({q_2})$), a contradiction to $p_1 q_1 \in E(G)$.
This completes the proof of Claim~E.

\2

{\bf Claim~F:} {\em ${\rm mac}(G) \geq \frac{1}{2} w(A_0) + \frac{7}{8} w(A_1) + w(A_2)$.}

\2

{\em Proof of Claim~F:} {For a set $U$ of vertices of $G$, let $G[U]$ denote the subgraph of $G$ induced by $U.$} Let $G_1^*,G_2^*,\ldots,G_r^*$ denote the connected components in $G^*$
and let $D_1^*,D_2^*, \ldots, D_r^*$ be the maximal subgraphs in $D^*$ such that $UG(D_i^*)=G_i^*$.
Let $F_i$ denote all edges in $E(G) \setminus E(G^*)$ with both endpoints in $V(G_i^*)$.
{Note that the edge set of $G[V(G_i^*)]$ is exactly $E(G_i^*) \cup E(F_i)$.
Furthermore note} that each $D_i^*$ is either an in-tree (an {\em  in-tree} is an oriented tree, where exactly one vertex, the root, has out-degree 0
and all other vertices have out-degree 1) or $D_i^*$ contains one directed cycle (possibly a 2-cycle) 
and all vertices have out-degree one in $D_i^*$.
We now consider the following cases (for $i=1,2,\ldots,r$):

\2

{\em Case F.1. $D_i^*$ is an in-tree.} 
If $f \in F_i$ then by Claim~C we note that any cycle in $G_i^* + f$ has length divisible by three.
As $G$ is triangle-free this implies that any odd cycle in $G_i^* + f$ has length at least nine.
Lemma~\ref{girth2} (with $k=4$) now implies the following.

\begin{equation}
\begin{array}{rcl}
{\rm mac}(G[V(G_i^*)]) & \geq & \frac{1}{2} (w(F_i)+w(G_i^*)) +  \frac{k-1}{2k} w(G_i^*)  \\
                       & \geq & \frac{1}{2} w(F_i) +  \frac{7}{8} w(G_i^*)  \\
\end{array}  \label{Eq:G1}
\end{equation}

\2

{\em Case F.2. $D_i^*$ contains a directed $2$-cycle.}  Let $C = c_1 c_2 c_1$ be the cycle in $D_i^*$.
As in the proof of Case~F.1 we note that if $f \in F_i$ then, by Claim~C, any cycle in $G_i^* + f$ has length divisible by three.
As $G$ is triangle-free this again implies that any odd cycle in $G_i^* + f$ has length at least nine.
Lemma~\ref{girth2} (with $k=4$) now implies the following, where we let $e^*$ in Lemma~\ref{girth2} be the edge $c_1 c_2$.

\begin{equation}
\begin{array}{rcl}
{\rm mac}(G[V(G_i^*)]) & \geq & \frac{1}{2} (w(F_i)+w(G_i^*)) +  \frac{k-1}{2k} w(G_i^*) + \frac{1}{2k} w(c_1c_2) \\
                       & \geq & \frac{1}{2} w(F_i) +  \frac{7}{8} w(G_i^*) + \frac{1}{8} w(c_1c_2) \\
\end{array}  \label{Eq:G2}
\end{equation}

\2

{\em Case F.3. $D_i^*$ contains a directed cycle of length greater than two. }
Let $C = c_1 c_2 \ldots c_l c_1$ be the cycle in $D_i^*$.
Note that by Claim~D, $l$ is divisible by 3 and $C$ contains no chord in $G$.
Note that every $u \in V(D_i^*) \setminus V(C)$ has a unique directed path from $u$ to $V(C)$
(as $G_i^*$ is connected and $\Delta^+(D_i^*) \leq 1$).
Let { ${\rm pe}(u)=c_j$} if and only if the unique path from $u$ to $V(C)$ ends in $c_j$.
Define $C_j$ as follows {(see Figure~\ref{ClaimF3fig} for an illustration).}
{
\[
C_j = \{ v \; | v \in V(D_i^*) \setminus V(C) \; \& \; {\rm pe}(v)=c_j \} \cup \{c_j\}
\]
}

\begin{figure}[hbtp]
\begin{center}
\tikzstyle{vertexDOT}=[scale=0.25,circle,draw,fill]
\tikzstyle{vertexY}=[circle,draw, top color=gray!5, bottom color=gray!30, minimum size=11pt, scale=0.8, inner sep=0.99pt]
\tikzstyle{vertexZ}=[circle,draw, top color=gray!5, bottom color=gray!30, minimum size=11pt, scale=0.5, inner sep=0.49pt]
\tikzstyle{vertexW}=[circle,draw, top color=gray!5, bottom color=gray!30, minimum size=11pt, scale=1.0, inner sep=1.1pt]

\begin{tikzpicture}[scale=0.5]
\node (ca) at (-2,2) [vertexY] {$c_1$};
\node (cb) at (2,2) [vertexY] {$c_2$};
\node (cc) at (5.5,-2) [vertexY] {$c_3$};
\node (cd) at (2.5,-2) [vertexY] {$c_4$};
\node (ce) at (-0.5,-2) [vertexY] {$c_5$};
\node (cf) at (-3.5,-2) [vertexY] {$c_6$};

\draw [->,line width=0.07cm] (ca) -- (cb);
\draw [->,line width=0.07cm] (cb) -- (cc);
\draw [->,line width=0.07cm] (cc) -- (cd);
\draw [->,line width=0.07cm] (cd) -- (ce);
\draw [->,line width=0.07cm] (ce) -- (cf);
\draw [->,line width=0.07cm] (cf) -- (ca);

\node (a1) at (-4,2) [vertexZ] {};
\node (a2) at (-6,2) [vertexZ] {};
\node (a3) at (-4,4) [vertexZ] {};
\node (a4) at (-2,4) [vertexZ] {};
\draw [->,line width=0.02cm] (a1) -- (ca);
\draw [->,line width=0.02cm] (a2) -- (a1);
\draw [->,line width=0.02cm] (a3) -- (a1);
\draw [->,line width=0.02cm] (a4) -- (ca);
\draw [rounded corners] (-7,1) rectangle (-1,5);   
\node [scale=1.2] at (-8,3) {$C_1$};

\node (b1) at (4,2) [vertexZ] {};
\node (b2) at (6,2) [vertexZ] {};
\node (b3) at (4,4) [vertexZ] {};
\draw [->,line width=0.02cm] (b1) -- (cb);
\draw [->,line width=0.02cm] (b2) -- (b1);
\draw [->,line width=0.02cm] (b3) -- (b1);
\draw [rounded corners] (1,1) rectangle (7,5);   
\node [scale=1.2] at (8,3) {$C_2$};

\draw [rounded corners] (4.5,-3) rectangle (6.5,-1);   
\node [scale=1.2] at (7.5,-2) {$C_3$};

\draw [rounded corners] (1.5,-3) rectangle (3.5,-1);   
\node [scale=1.2] at (2.5,-0.4) {$C_4$};

\draw [rounded corners] (-1.5,-3) rectangle (0.5,-1);   
\node [scale=1.2] at (-0.5,-0.4) {$C_5$};

\node (f1) at (-5.5,-2) [vertexZ] {};
\draw [->,line width=0.02cm] (f1) -- (cf);
\draw [rounded corners] (-6.5,-3) rectangle (-2.5,-1);   
\node [scale=1.2] at (-7.5,-2) {$C_6$};

\end{tikzpicture}
\end{center}
\caption{An illustration of the sets $C_i$'s in Claim~F.3.}
\label{ClaimF3fig}
\end{figure}
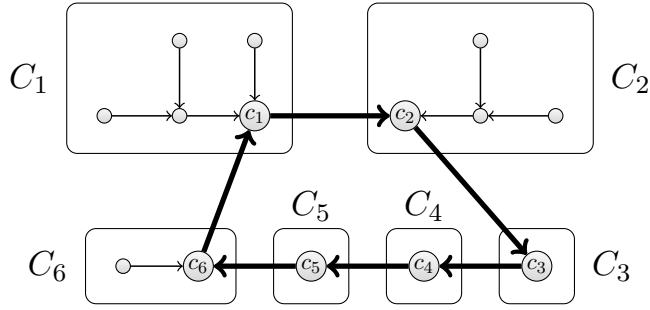

We will now show that if $u \in C_j$ and $v \in C_k,$ where $j \not= k$,  then $uv \not\in E(G) \setminus E(C)$.
For the sake of contradiction assume that $uv \in E(G) \setminus E(C)$. As $uv \not\in E(C)$ and $C$ has no chords in $G$
we note that $u \not= c_j$ or $v \not= c_k$. Without loss of generality assume that $u \not= c_j$. Let $P$ be the unique path
from $u$ to $c_j$ in $D_i^*$ and let $Q$ be the unique path from $v$ to $c_k$ followed by the path from $c_k$ to $c_j$ using the arcs of 
$C$. As $c_j c_{j+1} \in D_i^*$ (and $c_{j+1}c_j$ is not an arc in $D_i^*$, as $l \geq 3$, {and therefore also not an arc on $A(P) \cup A(Q)$}), Claim~E implies that $uv \not\in E(G)$, a contradiction.  Therefore $uv \not\in E(G) \setminus E(C)$,
as desired.

Also, $G_i^*[C_j]$ is a tree (possibly containing only one vertex) for all $j$. 
Analogously to Case~F.1 we note that, for every $f \in F_i$, any odd cycle in $G_i^*[C_j] + f$ has length at least nine.
Lemma~\ref{girth2} (with $k=4$) now implies the following:
\[
\begin{array}{rcl}
{\rm mac}(G[C_j]) & \geq & \frac{1}{2} (w(G^*_i[C_j]) + w(F_i \cap E(G[C_j])))  + \frac{k-1}{2k} w(G^*[C_j])  \\
                  & \geq & \frac{1}{2} w(F_i \cap E(G[C_j]))  + \frac{7}{8} w(G^*[C_j])  \\
\end{array}
\]

Recall that either $l$ is even or $l \geq 9$.
In this case picking an optimal weighted cut in each $C_j$ and adding all edges of $C$ if $|E(C)|$ is even or
all edges of $C$ except the cheapest one if $|E(C)|$ is odd, we obtain the following (as if $|E(C)|$ is odd then
$|E(C)| \geq 9$).
\[
{\rm mac}(G[V(G_i^*)])  \geq \frac{8}{9} w(C) + \frac{1}{2} w(F_i) + \frac{7}{8} w(E(G_i^*)\setminus E(C))\\
\]
As $8/9 > 7/8$ this implies the following: 

\begin{equation}
{\rm mac}(G[V(G_i^*)])  \geq  \frac{1}{2} w(F_i) + \frac{7}{8} w(G_i^*) \label{Eq:G3}
\end{equation}

This completes Case~F.3.

\2

{Note that every edge in $A_2 \cap E(G_i^*)$ is considered in Case~F.2 above and in this case Inequality~(\ref{Eq:G2}) holds.
Therefore any edge in $A_2 \cap E(G_i^*)$ is counted $\frac{7}{8} + \frac{1}{8}$ times. 
Now combining Inequality~(\ref{Eq:G1}),  Inequality~(\ref{Eq:G2}) and Inequality~(\ref{Eq:G3})
 we obtain the following.}

\[
{\rm mac}(G[V(G_i^*)])  \geq w(A_2 \cap E(G_i^*)) + \frac{1}{2} w(F_i) + \frac{7}{8}w(E(G_i^*) \setminus A_2)
\]

Let $(X_i,Y_i)$ be a maximum weight cut of $G_i^*$ and for
each $i \in [r],$ randomly and uniformly assign $X_i$ color 1 or 2 and $Y_i$ the opposite color.
Let $A$ be all vertices of color 1 and let $B$ be all vertices of color 2.
Now every edge in $(X_i,Y_i)$ lies in the cut induced by $(A,B)$ and every edge between different $G_i^*$'s lies in the cut induced by $(A,B)$
with probability 1/2. 
{Let $W$ be the weight of all edges between different $G_i^*$'s and note that the average weight of the cut $(A,B)$ is as follows.

\[
\frac{W}{2} + \sum_{i=1}^r {\rm mac}(G[V(G_i^*)])        
\]

If $e \in E(A_0)$ then $e$ either belongs to some $F_i$ or is an edge between different $G_i^*$'s (and in this case is counted in $W$), while
if $e \in E(A_1) \cup E(A_2)$ then $e$ belongs to some $G_i^*$. Therefore the following holds (as ${\rm mac}(G)$ is greater than or equal to the 
average weight of the cut $(A,B)$).
}

\[
{\rm mac}(G) \geq \frac{1}{2} w(A_0) + \frac{7}{8} w(A_1) + w(A_2)
\]

This completes the proof of Claim F.

\2

{\bf Claim~G:} {\em  ${\rm mac}(G) \geq \frac{3}{5} w(A_0) + \frac{3}{5} w(A_1) + w(A_2)$.}

\2

{\em Proof of Claim~G:} As $\Delta^+(D^*) \leq 1$ we note that all edges in $A_2$ form 
a matching in $G$ (i.e. they have no endpoints in common). 
By Lemma~\ref{mainMatch} (with $\Delta=3$ and $M=A_2$) 
we note that ${\rm mac}(G) \ge \frac{3}{5}(w(G) - w(A_2)) + w(A_2)$.
As $w(G)-w(A_2)=w(A_0)+w(A_1)$ this implies Claim~G.

\2

We now return to the proof of Theorem~\ref{mainProb}. By Claims~A, F and G, the following three inequalities hold.

\begin{description}
\item[(1)] ${\rm mac}(G) \ge w(A_0) + \frac{2}{3} w(A_1) + \frac{1}{3} w(A_2)$.
\item[(2)] ${\rm mac}(G) \geq \frac{1}{2} w(A_0) + \frac{7}{8} w(A_1) + w(A_2)$.
\item[(3)] ${\rm mac}(G) \geq \frac{3}{5} w(A_0) + \frac{3}{5} w(A_1) + w(A_2)$.
\end{description}

Taking $\frac{9}{22}$  times inequality (1) plus $\frac{8}{22}$  times  inequality (2) plus $\frac{5}{22}$  times inequality (3),
implies the following:
\[
\begin{array}{rcl}  \2
{\rm mac}(G) & \geq &  \left( \frac{9}{22} + \frac{8\cdot 1}{22 \cdot 2} + \frac{5 \cdot 3}{22 \cdot 5} \right) w(A_0) 
                     + \left( \frac{9\cdot 2}{22 \cdot 3} + \frac{8\cdot 7}{22 \cdot 8} + \frac{5 \cdot 3}{22 \cdot 5} \right) w(A_1) \\ \2 \2
            &  &         + \left( \frac{9\cdot 1}{22 \cdot 3} + \frac{8}{22} + \frac{5}{22} \right) w(A_2) \\
& = & \frac{16}{22} (w(A_0) + w(A_1) + w(A_2)) = \frac{8}{11} \cdot w(G) \\

\end{array}
\]
\end{proof}

Recall the statement of Theorem~\ref{mainProbTree}.

\2

\noindent
{\bf Theorem~\ref{mainProbTree}.} {\em Let $G$ be an edge-weighted triangle-free graph with $\Delta(G) \leq 3$
and let $T$ be an arbitrary spanning tree in $G$.
Then ${\rm mac}(G) \geq \frac{w(G)}{2} + 0.3193 \cdot w(T)$.
}

\begin{proof}
  By Lemma~\ref{lemmaProb}, with $p=0.85$ and $r=5$ and Theorem~\ref{mainProb}, we obtain the following inequalities.
\[
\begin{array}{crcl}
 \mbox{{\bf (a)}} & {\rm mac}(G) & \geq &  \frac{p+1}{2} w(T) + \frac{1-p^{r-1}}{2} (w(G)-w(T))  \\
                   &              & \geq & 0.925 \cdot w(T) + 0.23899687 (w(G)-w(T))\\
& & & \\
 \mbox{{\bf (b)}} & {\rm mac}(G) & \geq & \frac{8}{11} \cdot w(T) + \frac{8}{11} (w(G)-w(T)) \\
\end{array}
\]
Taking $0.46545$ times inequality (a) plus $0.53455$ times inequality (b) gives us the following inequality, which completes
the proof.
\[
{\rm mac}(G) \geq 0.8193 \cdot w(T) + \frac{w(G)-w(T)}{2}.
\]
\end{proof}

\section{Conclusion}\label{sec:c}

{In this paper, we study lower bounds of the maximum weight ${\rm mac}(G)$ of a cut in a weighted graph $G.$ We obtain lower bound for arbitrary graphs, graphs of bounded girth, and triangle-free graphs of bounded maximum degree. 
We posed a number of conjectures and an open problem. We conjecture that if $G$ is a weighted triangle-free graph and $T$ is a spanning tree of $G$, then ${\rm mac}(G)\ge \frac{w(G)}{2} + \frac{3w(T)}{8}.$
We also conjecture that ${\rm mac}(G)\ge 4w(G)/5$ for a weighted triangle-free subcubic graph $G$. Bondy and Locke \cite{BoLO86} proved  that the last conjecture holds for unweighted triangle-free subcubic graphs, in {other} words for weighted 
triangle-free subcubic graphs where each edge has the same weight.  The following conjecture is related to the main topic of the paper:
every triangle-free subcubic graph $G$ contains a set $E'$ of edges, such that every 5-cycle in $G$ contains exactly one edge from $E'.$}

\vspace{2mm}

\noindent {\bf Acknowledgement} We are grateful to Noga Alon for suggesting Theorem \ref{Shearer} and to the referees for numerous suggestions which significantly improved the presentation.

\end{document}